\RequirePackage{fix-cm} 
\documentclass[smallcondensed]{svjour3}     

\usepackage{amssymb} 
\usepackage{amsmath} 
\usepackage{amsfonts}

\usepackage{booktabs}
\usepackage{algorithm}%
\usepackage{algorithmicx}%
\usepackage{algpseudocode}%
\usepackage{appendix}

\usepackage{latexsym} 
\usepackage{mathptmx} 
\DeclareMathAlphabet{\mathcal}{OMS}{cmsy}{m}{n} 
\DeclareSymbolFont{largesymbols}{OMX}{cmex}{m}{n}
\usepackage{bm}       
\usepackage{bbm}      
\usepackage{graphicx} 
\usepackage{subfigure}
\usepackage{graphics} 
\usepackage{color}    
\usepackage{epsfig}   
\usepackage{epstopdf} 
\usepackage{algorithm}
\usepackage{algorithmicx}
\usepackage{mathrsfs} 
\usepackage{wasysym}  
\usepackage[colorlinks,linkcolor=blue,citecolor=blue]{hyperref} 
\usepackage{float}   
\usepackage{multirow} 
\usepackage{multicol} 
\usepackage{comment}  
\usepackage{overpic}  
\usepackage{psfrag}   
\usepackage{rotating} 
\usepackage{stmaryrd} 
\usepackage{verbatim} 
\usepackage{ifthen}
\usepackage{pifont}   
\smartqed             
\usepackage{enumitem} 

\newcounter{num}      
\numberwithin{equation}{section}

\numberwithin{theorem}{section}

\numberwithin{lemma}{section}

\numberwithin{corollary}{section}

\journalname{lxajournal}
\AtBeginDocument{%
  \paperwidth=\dimexpr
    1in + \oddsidemargin
    + \textwidth
    + 1in + \oddsidemargin
  \relax
  \paperheight=\dimexpr
    1in + \topmargin
    + \headheight + \headsep
    + \textheight
    + 1in + \topmargin
  \relax
  \usepackage{geometry}\relax
}

\usepackage[mathlines]{lineno}
\begin{document}
\title{Suboptimal subspace construction for log-determinant approximation}


\author{Zongyuan Han \and Wenhao Li \and Yixuan Huang  \and Shengxin Zhu}

\institute{
           Zongyuan Han
           \at School of Mathematical Sciences, Beijing Normal University and Laboratory of Mathematics and Complex Systems, Ministry of Education, Beijing 100875, P.R. China
   \and
          Wenhao Li
          \at Guangdong Provincial Key Laboratory of Interdisciplinary Research and Application for Data Science, Department of Applied Mathematics, BNU-HKBU United International College, Zhuhai 519087, P.R. China
    \and
           Yixuan Huang
           \at School of Mathematical Sciences, Beijing Normal University and Laboratory of Mathematics and Complex Systems, Ministry of Education, Beijing 100875, P.R. China
    \and
           {\large{\ding{41}}}Shengxin Zhu\\
          \email Shengxin.Zhu@bnu.edu.cn
           \at Research Center for Mathematics, Beijing Normal University, Zhuhai 519087, P.R. China; Guangdong Provincial Key Laboratory of Interdisciplinary Research and Application for Data Science, Department of Applied Mathematics, BNU-HKBU United International College, Zhuhai 519087, P.R. China
}

\date{Received: date / Accepted: date}
\maketitle
\begin{abstract}
Variance reduction is a crucial idea for Monte Carlo simulation and the stochastic Lanczos quadrature method is a dedicated method to approximate the trace of a matrix function. Inspired by their advantages, we combine these two techniques to approximate the log-determinant of large-scale symmetric positive definite matrices. Key questions to be answered for such a method are how to construct or choose an appropriate projection subspace and derive guaranteed theoretical analysis. This paper applies some probabilistic approaches including the projection-cost-preserving sketch and matrix concentration inequalities to construct a suboptimal subspace. Furthermore, we provide some insights on choosing design parameters in the underlying algorithm by deriving corresponding approximation error and probabilistic error estimations. Numerical experiments demonstrate our method's effectiveness and illustrate the quality of the derived error bounds. 




\end{abstract}

\section{Introduction}\label{sec1}
The computation of log determinants of symmetric positive definite matrices is a fundamental problem in high-dimensional inference. It has widespread applications in fields such as Gaussian process kernel learning \cite{Rasmussen2006gaussian}, linear mixed models \cite{Zhu2018Essential}, Markov random fields \cite{Wainwright2006Log}, Bayesian inference \cite{MacKay2003Information}, information geometry \cite{Amari2016Information}, and others. 

One straightforward approach to computing the logarithm of determinants (or log determinants) is through factorization. A sophisticated multi-frontal Cholesky decomposition approach can be used to calculate the log determinant and its derivatives for very large-scale sparse matrices \cite{Zhu2017Fast,zhu2019sparse}. Such a method, together with other matrix analyses, works efficiently for linear mixed models and has been well implemented in stable software. However, when the matrix is dense, the Cholesky decomposition requires a time complexity of $O(n^{3})$ and storage requirements of $O(n^{2})$, making it computationally prohibitive for large-scale problems \cite{golub2013matrix}. 

Alternatively, iterative methods can be designed according to the well-known identity \cite[p168]{Golub2009Matrices}, 
\begin{equation}
    \label{eq1}
    \log \det(A)=\sum_{i=1}^{n}\log(\lambda_{i})={\mathrm{tr}}(\log(A)),
\end{equation}
for symmetric positive definite matrices. After this transformation, the calculation of the log determinants can be reformulated as the problem of estimating the trace of matrix logarithm function $\log(A)$.

Several approaches have been proposed for estimating the trace of matrix logarithms. These approaches can be broadly categorized based on the techniques employed, including Monte Carlo-based methods, polynomial approximation-based methods, subspace iteration-based methods, and methods that utilize Gaussian quadrature and Lanczos iteration.

First, the matrix trace estimation method based on Monte Carlo can be traced back to Girard’s literature \cite{girard1989fast}, in which he proposed a fast Monte Carlo algorithm for approximating the calculation of matrix traces in minimizing Generalized Cross-Validation (GCV) problems. It can be simply described as 
\begin{equation}
    \frac{1}{n}\mathrm{tr}(A)=\mathbb{E}(\boldsymbol{z}^{T}A\boldsymbol{z}/\boldsymbol{z}^{T}\boldsymbol{z})\approx \frac{1}{m}\sum_{i=1}^{m}\frac{(\boldsymbol{z}^{(i)})^{T}A\boldsymbol{z}^{(i)}}{(\boldsymbol{z}^{(i)})^{T}\boldsymbol{z}^{(i)}},
\end{equation}
where $A\in \mathbb{R}^{n\times n}$, $\boldsymbol{z}$ and $\boldsymbol{z}^{(i)}$ are stochastic vectors (also referred to as query vectors) with entries independently sampled from the standard normal distribution, $m$ is the number of query vectors. Hutchinson trace estimator \cite{Hutchinson1990A} simplified the above estimation implementation and satisfies a minimum variance criterion, which can be expressed as
\begin{equation}
    \mathrm{tr} (A) = \mathbb{E}(\boldsymbol{z}^T A \boldsymbol{z}) \approx \frac{1}{m}\sum_{i=1}^m (\boldsymbol{z}^{(i)})^{T} A \boldsymbol{z}^{(i)},
\end{equation}
where $\boldsymbol{z}$ and $\boldsymbol{z}^{(i)}$ are stochastic query vectors with entries independently sampled from the Rademacher distribution. Subsequently, other estimators have been proposed for estimating the trace using the same form, but with different distributions for random vectors. These include random Gaussian vectors \cite{silver1997calculation}, random phase vectors \cite{iitaka2004random}, and columns derived from a Hadamard matrix \cite{bekas2007estimator}, among others. In contrast to previous research that focused solely on analyzing the variance of estimators, the seminal work of \cite{avron2011randomized} provides the first comprehensive analysis of the bound of query complexity for estimators, i.e., the minimum number of matrix-vector multiplications (MVM) required to achieve a desired accuracy and success rate. This bound was further improved and extended by Roosta-Khorasani and Ascher \cite{RA15}. Recently, multi-level Monte Carlo methods \cite{Giles2015Multilevel} have been introduced in \cite{Eric2022A,Andreas2022A} to accelerate the convergence rate of the stochastic trace estimator.

Second, some papers study the trace of matrix logarithms using a polynomial approximation of the logarithmic function. Boutsidis et al. applied Taylor series expansions to the logarithmic function and used the Monte Carlo method to approximate the traces of a small number of matrix powers \cite{Boutsidis2017A}. However, the Taylor series is hard to be optimal from the viewpoint of polynomial approximation. Therefore, Han et al. \cite{Han2015Large} proposed a stochastic Chebyshev expansion method to accelerate the Taylor approximation. Based on the eigenvalue distribution prior, Wang and Peng designed a weighted orthogonal polynomial approximation algorithm for log-determinant computation. The three-term-recursion relation of orthogonal polynomials makes the algorithm computationally efficient \cite{Peng2018A}.

Third, some subspace iteration methods are also used to estimate the trace and log determinants. The essence of such an approach is to retain the larger eigenvalues while dropping the smaller ones. To calculate $\log(\det(A))$ for $A\in \mathbb{R}^{n\times n}$, Saibaba et al. \cite{Saibaba2017Random} obtained a surrogate matrix $T\in \mathbb{R}^{l\times l} (l\ll n)$ of matrix $A$ using power iterations and QR factorization. Compared to Monte Carlo-based trace estimators, this estimator can be computationally efficient when $A$ is a low-rank matrix. Similarly, Li et al. \cite{Li2021Randomized} obtained the surrogate matrix $T$ using the block Krylov subspace method. Theoretical analysis and numerical experiments show that they achieve better error bounds than those in \cite{Saibaba2017Random}.

Fourth, Bai and Golub \cite{bai1996bounds} presented deterministic upper and lower bounds for $\mathrm{tr}(A^{-1})$ and $\log\det(A)$ applying Gaussian quadrature and related theory. They also proposed \cite{Bai1996some} probabilistic upper and lower bounds for large sparse matrices using integral approximation and Hoeffding’s inequality, the estimation process features the Hutchinson trace estimator and the Lanczos method to construct a Gaussian quadrature, which is currently known as the stochastic Lanczos quadrature (SLQ).  In \cite{Ubaru2017Fast}, Ubaru et al. named the SLQ method for estimating $\mathrm{tr}(f(A))$ and provides the rigorous theoretical analysis that is missing in \cite{Bai1996some}. \cite{li2023analysis} further analyzes SLQ in the context of asymmetric quadrature nodes. In the following text, we will discuss the relevant details of the SLQ method in-depth and in relation to our problems.

Recently, Meyer et.al \cite{hutch++} proposed a stochastic trace estimator named Hutch++, which combined the Hutchinson trace estimator with subspace projection. This new estimator can be considered as a variance reduction version (also used e.g. in \cite{gambhir2017deflation,lin2017randomized,baston2022stochastic}) of the Hutchinson trace estimator \cite{Hutchinson1990A}, it reduced the query complexity from $O(1/\epsilon^{2})$ (as in \cite{avron2011randomized,RA15}) to $O(1/\epsilon)$. Later on, \cite{persson2022improved} developed an adaptive version of Hutch++, which near-optimally divides the MVM between the two phases of stochastic trace estimation and subspace projection. Similarly, \cite{chen2022krylov} proposed a variance reduction trace estimation scheme that adaptively allocates the numbers of Block Krylov subspace iterations and query vectors. Different from Hutch++, \cite{epperly2023xtrace} designed an exchangeable trace estimator named XTrace, which exploited both variance reduction and the exchangeability principle, in short, the test vectors for low-rank approximation and for estimating the trace of the residual are the same.

Unlike \cite{Ubaru2017Fast}, we combine the variance reduction technique with the SLQ method  to estimate the trace of matrix logarithm. Our method contains four key ingredients. The first is the Hutchinson trace estimator, used to estimate the trace. The second is the stochastic Lanczos quadrature method, used to efficiently approximate quadratic forms.  The third are concentration inequalities, including Markov and Hanson-Wright inequalities. The fourth is the projection-cost-preserving sketch, used to bound the size of the Gaussian random matrix and the number of Hutchinson query vectors. This is one of the significant differences from \cite{Ubaru2017Fast}. Furthermore, we explicitly present the bounds of all design parameters involved in our method, such as the dimension of the projection subspace, the number of Hutchinson query vectors, and the iterations of the Lanczos procedure. This provides practitioners with more guidance than asymptotic bounds in the form of $O(\cdot)$ or $\Omega(\cdot)$.

The paper is organized as follows. Section \ref{sec:preli} presents the preliminaries. In Section \ref{sec:approx}, we provide the main idea for approximating $\mathrm{tr}(\log(A))$ and state the main theorem. Section \ref{sec:error} presents an error analysis of the relative probability error bound and explicit bounds for the relevant parameters. The performance evaluation of our method is given in Section \ref{sec:performance}. Finally, we present our concluding remarks.

\section{Preliminaries}
\label{sec:preli}
The main result of this paper is given by Theorem \ref{the_main_result} and the following figure shows the roadmap for proving Theorem \ref{the_main_result}.

\begin{figure}[H]
    \centering	
    \includegraphics[width=.6\textwidth]{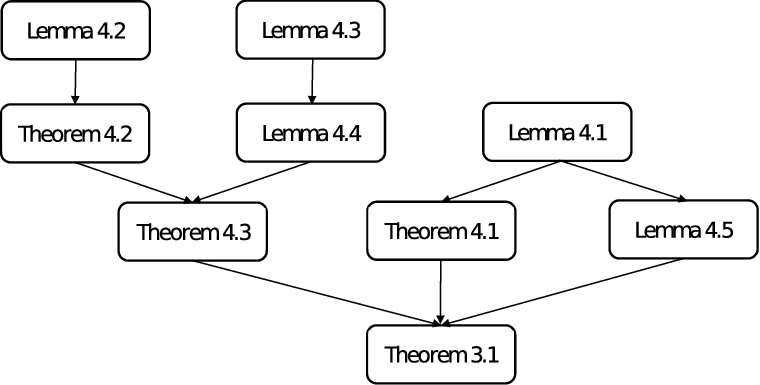}
    \caption{The roadmap for proving Theorem \ref{the_main_result}. }
\end{figure}
Let $\|\cdot\|_{2}$ denote the Euclidean norm for a vector and the spectral norm for a matrix. Let $\|\cdot\|_{F}$ denote the Frobenius norm (F-norm) for a matrix. Let $A\in \mathbb{R}^{n\times n}$ be a symmetric positive definite (SPD) matrix with eigendecomposition $A=U\Lambda U^{T}$, where $U\in \mathbb{R}^{n\times n}$ is orthogonal and $\Lambda = {\mathrm{diag}}(\lambda_{1},\ldots, \lambda_{n})\in \mathbb{R}^{n\times n}$ is a diagonal matrix of eigenvalues with non-increasing ordering. Let $A_{k}=\arg\min _{{\mathrm{rank}}(B)= k}\|A-B\|_{F}$ be the best rank-$k$ approximation to $A$. Matrix function $f(A)\triangleq\log(A)$, and its best rank-$k$ approximation is denoted by $A_{k}(f)={\arg\min}_{\mathrm{rank}(B)=k}\|f(A)-B\|_{F}$. By the Eckart-Young theorem \cite{Eckart1936The}, $A_{k}(f)=U_{k}U_{k}^{T}f(A)$ where $U_{k}$ contains the first $k$ columns of $U$.

To prepare for the subsequent analysis, we will introduce some related definitions and lemmas.

If matrix $S$ can project a high-dimensional point cloud $E$ onto a lower-dimensional space while approximately preserving vector norms, then $S$ is called a Johnson-Lindenstrauss Embedding (JLE). A more detailed description is as follows.

\begin{definition}(JLE \cite{Johnson1986Extensions,Bamberger2021Optimal})
    Let $S\in \mathbb{R}^{k\times n}$ be a random matrix where $k<n$, $p\in \mathbb{N}$ and $\epsilon, \delta\in (0,1)$. We say that $S$ is a $(p,\epsilon, \delta)$-JLE if for any subset $E\subseteq \mathbb{R}^{n}$ with $|E|=p$ and
    \begin{linenomath}
    $$(1-\epsilon)\|\boldsymbol{x}\|_{2}^{2}\leq \|S\boldsymbol{x}\|_{2}^{2} \leq (1+\epsilon)\|\boldsymbol{x}\|_{2}^{2},$$
    \end{linenomath}
    holds simulatneously for all $\boldsymbol{x}\in E$ with a probability of at least $1-\delta$.
\end{definition}

Furthermore, if $S$ is a JLE of the range of a matrix $M$, then $S$ is called the subspace embedding of $M$, the specific definition is as follows.

\begin{definition}(Subspace embedding \cite{Cameron2020Proj})
    $S\in \mathbb{R}^{k\times m}$ is an $\epsilon$-subspace embedding for $M\in \mathbb{R}^{m\times n}$, if $\forall \boldsymbol{x}\in \mathbb{R}^{n}$,
     \begin{linenomath}
    $$(1-\epsilon)\|M\boldsymbol{x}\|_{2}^{2}\leq \|SM\boldsymbol{x}\|_{2}^{2}\leq (1+\epsilon)\|M\boldsymbol{x}\|_{2}^{2}.$$
    \end{linenomath}
\end{definition}

Similar to JLE, Kane and Nelson \cite{Daniel2014Sparser} defined a high-order moment form of norm preservation for a projected vector, known as JL moments.

\begin{definition}(JL moments \cite{Cameron2020Proj,Daniel2014Sparser})
    Matrix $S\in \mathbb{R}^{n\times q}$ satisfies the $(\epsilon,\delta,\ell)$-JL moments if for any $\boldsymbol{x}\in \mathbb{R}^{n}$ with $\|\boldsymbol{x}\|_{2}=1$,
    \begin{linenomath}
$$\mathbb{E}_{S}\left|\|\boldsymbol{x}^{T}S\|_{2}^{2}-1\right|^{\ell}\leq \epsilon^{\ell}\cdot \delta.$$
\end{linenomath}
\end{definition}

For any matrix $A\in \mathbb{R}^{n\times d}$, if there is a matrix $\tilde{A}\in \mathbb{R}^{n\times d'} (d'\ll d)$ that preserves the distance between $A$'s columns and any $k$-dimensional subspace, then $\tilde{A}$ can be used as a surrogate for $A$ to solve certain low-rank optimization problems. $\tilde{A}$ is called a projection-cost-preserving sketch (PCPS) of $A$. A more detailed description is as follows. 

\begin{definition}(PCPS \cite{Michael2015Dimen})
    \label{projection_cost_preserving_sketch}
    $\tilde{A}\in \mathbb{R}^{n\times d'}$ is a rank-$k$ PCPS of $A\in \mathbb{R}^{n\times d}$ $(d\gg d')$ with error $0\leq \epsilon < 1$, if for all rank-$k$ orthogonal projection matrices $P\in \mathbb{R}^{n\times n}$,
    \begin{equation}
		(1-\epsilon)\|A-PA\|_{F}^{2}\leq \|\tilde{A}-P\tilde{A}\|_{F}^{2}+c\leq (1+\epsilon)\|A-PA\|_{F}^{2},
    \end{equation}
    for some fixed non-negative constant $c$ that may depend on $A$ and $\tilde{A}$ but independent of $P$. $\tilde{A}$ is also called an $(\epsilon, c, k)$-PCPS of $A$.
\end{definition}

The constant $c$ in Definition \ref{projection_cost_preserving_sketch} is used to control the tightness of the bound or the similarity between $A$ and $\tilde{A}$. In this paper, we set $c$ to $0$ to get a tight bound. 

Next, we present two important inequalities that will be used later in this paper, along with some remarks.

\begin{lemma}(\cite{hutch++})
    \label{k_rank_app}
    Let $A$ be a symmetric positive semidefinite matrix and $A_{k}$ be the best rank-$k$ approximation of $A$. Then, $\|A-A_{k}\|_{F}\leq \mathrm{tr}(A)/\sqrt{k}$.
\end{lemma}

\begin{lemma}(Hanson-Wright inequality \cite{hanson1971bound,Jelani2015Dimen})
    \label{hanson_wright_inequality}
    Let $\boldsymbol{z}=(z_{1},\ldots,z_{n})^{T}$ represent a vector with entries that are independently sampled from a Rademacher distribution and $A\in \mathbb{R}^{n\times n}$ be a real and symmetric matrix, for all $p\geq 1$
    \begin{linenomath}
    \begin{equation*}
    \|\boldsymbol{z}^{T}A\boldsymbol{z}-\mathbb{E}\boldsymbol{z}^{T}A\boldsymbol{z}\|_{p}\lesssim \sqrt{p}\cdot\|A\|_{F}+p\cdot \|A\|_{2},
    \end{equation*}
    \end{linenomath}
    where $\|\cdot\|_{p}$ denotes $(\mathbb{E}|\cdot|^{p})^{1/p}$.
\end{lemma}

\begin{remark}
    In this paper, we suppose matrix $A\in \mathbb{R}^{n\times n}$ is an SPD matrix with $\lambda_{\min}(A)=\lambda_{n}\geq 1$.  If $\lambda_{\min}(A) < 1$, one may let $\hat{A} \triangleq A/\lambda_{\min}(A)$, whose minimum eigenvalue is greater than or equal to $1$, indicating that
    \begin{equation}
        \label{eq2}
        \log\det(A)=\log\det(\lambda_{\min}(A)\hat{A})=n\log(\lambda_{\min}(A))+{\mathrm{tr}}(\log(\hat{A})),
    \end{equation}
    then $\mathrm{tr}(\log(\hat{A}))$ is the crucial problems as we considered.
\end{remark}

\begin{remark}
     Moreover, if $A$ is asymmetric and non-singular, then its logarithmic determinant (absolute value) can be obtained through the log-determinant problem of SPD matrix $A^{T}A$ since
     \begin{linenomath}
    $$
    \log (|\det(A)|) = \frac{1}{2}\log(\det(A)^{2}) = \frac{1}{2}\log (\det(A^{T}A)),
    $$
    \end{linenomath}
    which is transformed into the problem described in (\ref{eq2}).
\end{remark}

\section{Approximation of tr(f(A))}
\label{sec:approx}
Let $S\in \mathbb{R}^{n\times q}$ be a random matrix whose entries are independent normal random variables, and $Q\in \mathbb{R}^{n\times k}$ consist of $k$ principal orthonormal bases of the column space spanned by $f(A)S$.

Enlighted by the variance reduction technique introduced in \cite{hutch++}, we separate $f(A)$ into its projection onto the subspace spanned by $Q$ and its orthogonal complement, that is,
\begin{linenomath}
\begin{equation}
    \label{eq3}
    \begin{array}{ll}
		{\mathrm{tr}}\left(f(A)\right) &={\mathrm{tr}}(QQ^{T}f(A))+{\mathrm{tr}}((I-QQ^{T})f(A))\\
		&={\mathrm{tr}}(Q^{T}f(A)Q)+{\mathrm{tr}}((I-QQ^{T})f(A)(I-QQ^{T}))\\
		&=\underbrace{{\mathrm{tr}}(Q^{T}f(A)Q)}_{\text{first part}}+\underbrace{{\mathrm{tr}}(\Delta)}_{\text{second part}}. \\
    \end{array}
\end{equation}
\end{linenomath}
Let $\Delta \triangleq (I-QQ^{T})f(A)(I-QQ^{T})$. The second equality holds due to the cyclic property of the trace and the idempotency of $I-QQ^{T}$, that is, $I-QQ^{T}=(I-QQ^{T})^{2}$.

These two parts in (\ref{eq3}) will be estimated separately. For the first part, we will use the $(m+1)$-step Lanczos quadrature approximation method to estimate the trace. For the second part, we will combine the $(m+1)$-step Lanczos quadrature approximation method with the $N$-query Hutchinson trace estimator to estimate the trace of $\Delta$. The final sum of these two estimated results will be represented by the symbol $\Gamma$.

The following theorem is the main result of this paper, which is similar in form to Theorem 4.1 in \cite{Ubaru2017Fast}. As our trace estimation method adopts the variance reduction technique, new design parameters $k$ and $q$ are included, and the explicit lower bounds for all these parameters are presented in the theorem.

\begin{theorem}
    \label{the_main_result}
    Given $\epsilon,\delta \in (0,1)$, an SPD matrix $A\in \mathbb{R}^{n\times n}$ with its minimum eigenvalue $\lambda_{\min}\geq 1$. Let $S\in R^{n\times q}$ be a random matrix and $Q\in \mathbb{R}^{n\times k}$ be composed of the $k$ principal orthonormal bases of the column space spanned by $f(A)S$. Thus, if the following inequality is satisfied
    \begin{itemize}
		\item the column number of $Q$ satisfies $k\geq 16(1+\epsilon)/(1-\epsilon)$,
		\item the column number of $S$ satisfies $q\geq 288k/(\epsilon^{2}\delta)$,
		\item the number of query vectors $N \geq \left(1+\sqrt{1+4\epsilon\sqrt{\delta}}\right)^{2}/(2\epsilon^{2}\delta),$
         \item the Lanczos iteration parameter $m'\geq \frac{1}{2}\log(2kC_{\rho}/\varepsilon)/\log(\rho)$ for the first part and  $m\geq \frac{1}{2}\log(4nC_{\rho}/\varepsilon)/\log(\rho)$ for the second part, where $C_{\rho}=4M_{\rho}/(C(\rho^{2}-\rho))$ and $C=(n-1)f(\lambda_{\min})+f(\lambda_{\max})$, $M_{\rho}=|\log(\lambda_{\min}/2)|+\pi$ and 
  \begin{linenomath}
  $$\rho=(\lambda_{\max}+\sqrt{2\lambda_{\max}\lambda_{\min}-\lambda_{\min}^{2}})/(\lambda_{\max}\lambda_{\min}),$$
  \end{linenomath}
    \end{itemize}   
    then
    \begin{linenomath}
    $$\mathbb{P}\left\{\left|\mathrm{tr}(f(A))-\Gamma\right|\leq \epsilon |\mathrm{tr}(f(A))|\right\}\geq 1-\delta,$$
    \end{linenomath}
    where $\Gamma$ is the final estimation of $\mathrm{tr}(f(A))$.
\end{theorem}

The proof of this theorem is deferred to Section \ref{subsec:total}. To prove the theorem, we first derive the approximation error bound of the second part in \eqref{eq3}, followed by the error bound of the first part.

\section{The approximate error analysis}
\label{sec:error}
To approximate $\mathrm{tr}(\Delta)$ in (\ref{eq3}), we employ the SLQ method proposed by \cite{Ubaru2017Fast}. The approximation process is divided into two stages. 

In the first stage, ${\mathrm{tr}}(\Delta)$ is approximated by the Hutchinson trace estimator with $N$-query random Rademacher vectors $Z=[\boldsymbol{z}^{(1)},\ldots,\boldsymbol{z}^{(N)}]$. We use the notation $H^{N}(\Delta)$ to represent this approximation result,
\begin{equation}
    \label{eq4}
    H^{N}(\Delta) \triangleq  \frac{1}{N}\sum\limits_{i=1}^{N}(\boldsymbol{z}^{(i)})^{T}\Delta\boldsymbol{z}^{(i)} \approx {\mathrm{tr}}\left(\Delta \right).
\end{equation}
Let $\boldsymbol{v}^{(i)}=\frac{(I-QQ^{T})\boldsymbol{z}^{(i)}}{\|(I-QQ^{T})\boldsymbol{z}^{(i)}\|_{2}}$ and $\boldsymbol{\mu}^{(i)}=U^{T}\boldsymbol{v}^{(i)}=[\mu_{1}^{(i)},\ldots,\mu_{n}^{(i)}]^{T}$. Recall that $A=U\Lambda U^{T}$, then
\begin{linenomath}
\begin{align*}
		H^{N}(\Delta) &=\frac{1}{N}\sum\limits_{i=1}^{N}\|(I-QQ^{T})\boldsymbol{z}^{(i)}\|_{2}^{2}(\boldsymbol{v}^{(i)})^{T}f(A)\boldsymbol{v}^{(i)}\\
		&=\frac{1}{N}\sum\limits_{i=1}^{N}\|(I-QQ^{T})\boldsymbol{z}^{(i)}\|_{2}^{2}(\boldsymbol{v}^{(i)})^T U f(\Lambda) U^T \boldsymbol{v}^{(i)} \\ 
		&=\frac{1}{N}\sum\limits_{i=1}^{N}\|(I-QQ^{T})\boldsymbol{z}^{(i)}\|_{2}^{2}\sum\limits_{j=1}^{n}f(\lambda_{j})({\mu_{j}^{(i)}})^{2}\\
		& = \frac{1}{N}\sum\limits_{i=1}^{N}\|(I-QQ^{T})\boldsymbol{z}^{(i)}\|_{2}^{2}\int_{\lambda_{n}}^{\lambda_{1}}f(t)d\mu^{(i)}(t).
\end{align*}
\end{linenomath}
In the last equality, we view $\sum_{j=1}^{n}f(\lambda_{j})(\mu_{j}^{(i)})^{2}$ as a Riemann-Stieltjes integral with the form of $\int_{\lambda_{n}}^{\lambda_{1}}f(t)d\mu^{(i)}(t)$, where the measure $\mu^{(i)}(t)$ is a piecewise constant function defined as \cite[pp. 9]{Golub2009Matrices}
\begin{equation}
    \label{measure_function}
    \mu^{(i)}(t)=\left\{\begin{array}{ll}
    0, &\quad \mbox { if } t<\lambda_{n}, \\
    \sum_{j=i}^{n} (\mu_{j}^{(i)})^{2}, &\quad \mbox { if } \lambda_{i} \leq t<\lambda_{i-1}, i=2, \ldots, n, \\
    \sum_{j=1}^{n} (\mu_{j}^{(i)})^{2}, &\quad \mbox { if } \lambda_{1} \leq t.
    \end{array}\right.
\end{equation}

As a Riemann-Stieltjes integral $\int_{\lambda_{n}}^{\lambda_{1}}f(t)d\mu(t)$ can be estimated by applying the Gaussian quadrature rule, that is
\begin{equation}
    \label{gauss_quadrature_approx}	            
    \boldsymbol{v}^{T}f(A)\boldsymbol{v}=\int_{\lambda_{n}}^{\lambda_{1}}f(t)d\mu(t)\approx \sum_{k=0}^{m}\tau_{k}f(\theta_{k}),
\end{equation}
where $\{\tau_{k}\}$ are the weights and $\{\theta_{k}\}$ are the nodes of the $(m+1)$-point Gaussian quadrature \cite{golub1969calculation}. Moreover, let $T_{m+1}$ be the Jacobi matrices derived by $(m+1)$-steps Lanczos algorithm \cite[p.39]{Golub2009Matrices} for matrix $A$ and any given vector $\boldsymbol{v}$, the quadrature weights $\tau_{k}=[\boldsymbol{e}_{1}^{T}\boldsymbol{y}_{k}]^{2}$ and nodes $\theta_k$ are extracted from the eigenpairs $\{(\theta_{k},\boldsymbol{y}_{k}),k=0,1,\ldots,m\}$ of $T_{m+1}$. The approximate method described by the formula \eqref{gauss_quadrature_approx} is called $(m+1)$-steps Lanczos quadrature method.

Thus, in the second stage, we approximate $H^{N}(\Delta)$ with $(m+1)$-steps Lanczos quadrature method, and use the notation $L_{m+1}^{N}(\Delta)$ to represent this approximation result,
\begin{equation}
    L_{m+1}^{N}(\Delta) \triangleq  \dfrac{1}{N}\sum_{i=1}^{N}\|(I-QQ^{T})\boldsymbol{z}^{(i)}\|_{2}^{2}\sum_{k=0}^{m}\tau_{k}^{(i)}f(\theta_{k}^{(i)}) \approx  H^{N}(\Delta).
\end{equation}

There are two approximation errors associated with the calculation of $\mathrm{tr}(\Delta)$, corresponding to the two stages of the approximation process. The first stage error, known as the Lanczos quadrature approximation error, arises from the difference between $H^{N}(\Delta)$ and $L^{N}_{m+1}(\Delta)$ and is given by $|H^{N}(\Delta)-L_{m+1}^{N}(\Delta)|$. The second stage error, known as the Hutchinson estimation error, arises from the difference between $\mathrm{tr}(\Delta)$ and $H^{N}(\Delta)$ and is given by $|H^{N}(\Delta)-\mathrm{tr}(\Delta)|$. We will analyze each of these errors in the following subsections.

\subsection{Lanczos quadrature approximation error}
\label{subsec:Lanczos}
As $\|(I-QQ^{T})\boldsymbol{z}^{(i)}\|_{2}^{2}\leq \|I-QQ^{T}\|_{2}^{2}\|\boldsymbol{z}^{(i)}\|_{2}^{2}=n$, the Lanczos quadrature approximation error is bounded by 
\begin{linenomath}
    \begin{align*}
		&\left|H^{N}(\Delta)-L_{m+1}^{N}(\Delta)\right|\\
		&=\left|\dfrac{1}{N}\sum\limits_{i=1}^{N}\|(I-QQ^{T})\boldsymbol{z}^{(i)}\|_{2}^{2}\left((\boldsymbol{v}^{(i)})^{T}f(A)\boldsymbol{v}^{(i)}-\sum\limits_{k=0}^{m}\tau_{k}^{(i)}f(\theta_{k}^{(i)})\right)\right|\\
		&\leq \dfrac{n}{N}\sum\limits_{i=1}^{N}\left|(\boldsymbol{v}^{(i)})^{T}f(A)\boldsymbol{v}^{(i)}-\sum\limits_{k=0}^{m}\tau_{k}^{(i)}f(\theta_{k}^{(i)})\right|.
\end{align*}
\end{linenomath}
Before delving into the analysis of the above error bound, we first consider a standard case where function $g$ is defined in $[-1,1]$, then generalize the result to the function $f$ defined in $[\lambda_{n},\lambda_{1}]$ with an affine linear transform.

\begin{lemma}{\cite[Theorem 3]{Alice2021On} and \cite[Theorem 3.3]{li2023analysis}} 
    \label{lemma_1}
    Let function $g$ be analytic in $[-1,1]$ and analytically continuable in the open Bernstein ellipse $E_{\rho}$ with foci $\pm 1$ and the sum of major and minor axis equals $\rho>1$, where it satisfies $|g(z)|\leq M_{\rho}$. Then the $(m+1)$-step Lanczos quadrature approximation error $E_{m+1}(g)$ satisfies
    \begin{equation}
        E_{m+1}(g)\triangleq\left|\int_{-1}^{1}g(t)d\mu(t)-\sum_{k=0}^{m}\tau_{k}g(\theta_{k})\right|\leq \frac{4M_{\rho}}{(\rho-1)\rho^{2m+1}},		
    \end{equation}
    where the $\mu(t)$ is the measure of integration, $\{\tau_{k}\}$ and $\{\theta_{k}\}$ are the weights and nodes of Gaussian quadrature rule and derived by $(m+1)$-steps Lanczos algorithm.
\end{lemma}

For the proof of this lemma, readers are referred to \cite{Ubaru2017Fast}, and for a minor correction to the result, readers may consult \cite{Alice2021On,li2023analysis}.

\begin{theorem}
    \label{the_Lanczos_error}
    Let $A\in \mathbb{R}^{n\times n}$ be an SPD matrix with minimum eigenvalue $\lambda_{\min}\geq 1$,  $f(t)$ is analytic on $[\lambda_{\min},\lambda_{\max}]$. If the Lanczos iteration parameter $m$ satisfies 
    \begin{linenomath}
        $$m\geq \frac{1}{2}\log\frac{16nM_{\rho}}{C\varepsilon(\rho^{2}-\rho) }\bigg/\log(\rho),$$
    \end{linenomath}
    where $C=(n-1)f(\lambda_{\min})+f(\lambda_{\max})$, $M_{\rho}=|\log(\frac{\lambda_{\min}}{2})|+\pi$ and 
    $\rho=\frac{\lambda_{\max}+\sqrt{2\lambda_{\max}\lambda_{\min}-\lambda_{\min}^{2}}}{\lambda_{\max}\lambda_{\min}}$,
    then 
    \begin{linenomath}
        $$|H^{N}(\Delta)-L_{m+1}^{N}(\Delta)|\leq \frac{\varepsilon}{4}|{\mathrm{tr}}\left(f(A)\right)|.$$
    \end{linenomath}
\end{theorem}

\begin{proof}
Define a new function
\begin{linenomath}
    $$g(t)\triangleq f\left[\left(\frac{\lambda_{\max}-\lambda_{\min}}{2}\right)t+\left(\frac{\lambda_{\max}+\lambda_{\min}}{2}\right)\right],$$
\end{linenomath}
which is analytic on $[-1,1]$ and has singularity at $t=-(\lambda_{\max}+\lambda_{\min})/(\lambda_{\max}-\lambda_{\min})$. To ensure that $g(t)$ can be analytically continued to an appropriate Bernstein ellipses $E_{\rho}$ with foci $\pm 1$, one can choose the length of semi-major axis as $\alpha=\lambda_{\max}/(\lambda_{\max}-\lambda_{\min})$, then the length of semi-minor axis is $\beta=\sqrt{\alpha^{2}-1}$, so the elliptical radius $\rho$ can be derived by
\begin{linenomath}
    $$\rho = \alpha+\beta = \frac{\lambda_{\max}+\sqrt{2\lambda_{\max}\lambda_{\min}-\lambda_{\min}^{2}}}{\lambda_{\max}-\lambda_{\min}}>1.$$
\end{linenomath}
As $g(t)$ is analytic on the open ellipse $E_{\rho}$ defined above,
\begin{linenomath}
\begin{align*}
    \max\limits_{t\in E_{\rho}} |g(t)| & =\max\limits_{t\in E_{\rho}}\left|\log\left(\frac{\lambda_{\max}-\lambda_{\min}}{2}t+\frac{\lambda_{\max}+\lambda_{\min}}{2}\right)\right|\\
    &\leq \max\limits_{t\in E_{\rho}} \sqrt{\left(\log\left|\frac{\lambda_{\max}-\lambda_{\min}}{2}t+\frac{\lambda_{\max}+\lambda_{\min}}{2}\right|\right)^{2}+\pi^{2}}\\
    &=\sqrt{\left(\log\left|\frac{\lambda_{\max}-\lambda_{\min}}{2}(-\frac{\lambda_{\max}}{\lambda_{\max}-\lambda_{\min}})+\frac{\lambda_{\max}+\lambda_{\min}}{2}\right|\right)^{2}+\pi^{2}}\\
    &\leq \left|\log(\frac{\lambda_{\min}}{2})\right|+\pi \ \triangleq \ M_{\rho},
\end{align*}
\end{linenomath}
where the first inequality results from $|\log(z)|=|\log |z|+i\arg(z)|\leq \sqrt{(\log|z|)^{2}+\pi^{2}}$, and the second equality holds since the maximum absolute value of the logarithm on the ellipse is attained at the endpoint $t=-\lambda_{\max}/(\lambda_{\max}-\lambda_{\min})$ on the real axis \cite[p. 175-178]{brown2009complex} and \cite[Corollary 5]{Alice2021On}. 

Next, based on Lemma \ref{lemma_1} and Corollary 3 in \cite{Alice2021On}, we obtain
\begin{linenomath}
    \begin{align*}
    |H^{N}(\Delta)-L_{m+1}^{N}(\Delta)|&\leq \dfrac{n}{N}\sum\limits_{i=1}^{N}\left|(\boldsymbol{v}^{(i)})^{T}f(A)\boldsymbol{v}^{(i)}-\sum\limits_{k=0}^{m}\tau_{k}^{(i)}f(\theta_{k}^{(i)})\right|\\
    &= \dfrac{n}{N} \sum\limits_{i=1}^{N}\left|\int_{\lambda_{\min}}^{\lambda_{\max}}f(t)d\mu^{(i)}(t)-\sum\limits_{k=0}^{m}\tau_{k}^{(i)}f(\theta_{k}^{(i)})\right|\\
    & \leq \frac{4nM_{\rho}}{(\rho -1)\rho^{2m+1}},
\end{align*}
\end{linenomath}
where the last inequality utilizes $E_{m+1}(f)=E_{m+1}(g)$ \cite[Proposition 3.4]{li2023analysis}.

As $f(t)$ increases monotonically on the interval $[\lambda_{\min},\lambda_{\max}]$, one may obtain
\begin{linenomath}
   $$C\triangleq [(n-1)f(\lambda_{\min})+f(\lambda_{\max})]\leq \sum_{i=1}^{n}f(\lambda_{i})={\mathrm{tr}}(f(A)).$$
\end{linenomath}
Let $C_{\rho}\triangleq \frac{4M_{\rho}}{C(\rho^{2}-\rho)}$, when the Lanczos iteration parameter $m$ satisfies
\begin{equation}
    \label{the_bound_of_m}
    m\geq \frac{1}{2}\log\left(\frac{4nC_{\rho}}{\varepsilon }\right)/\log(\rho),
\end{equation}
the Lanczos quadrature approximation error has the following upper bound
\begin{equation}
    |H^{N}(\Delta)-L_{m+1
    }^{N}(\Delta)|\leq \frac{\varepsilon}{4}C \leq \frac{\varepsilon}{4}{\mathrm{tr}}(f(A)).
\end{equation}
\qed
\end{proof}

\subsection{Hutchinson estimation error with bound of F-norm}
\label{subsec:hutch1}
In this subsection, we analyze the error bound of the Hutchinson trace estimator for any symmetric matrix $A$. \cite{hutch++} showed that when the number of query vectors for the Hutchinson trace estimator is $N=O(\log(1/\delta)/\epsilon^{2})$, then $\left|H^{N}(A)-\mathrm{tr}(A)\right|\leq \epsilon \|A\|_{F}$ with probability at least $1-\delta$. However, this bound for $N$ lacks explicit constants, providing little guidance for practitioners. To address this issue, we derive an explicit bound for the number of query vectors in the following discussion.

\begin{lemma}
\label{single_sample_inequality}
Let $A\in \mathbb{R}^{n\times n}$ be a symmetric matrix, $\boldsymbol{z}=(z_{1},\ldots,z_{n})^{T}$ be a Rademacher random vector. Then, for all $\epsilon >0$
\begin{equation}
    \mathbb{P}\left\{ |\boldsymbol{z}^{T}A\boldsymbol{z}-{\mathrm{tr}}(A)|\geq \epsilon \|A\|_{F}\right\}\leq \frac{(\sqrt{2}+2\mathrm{sr}(A)^{-\frac{1}{2}})^{2}}{\epsilon^{2}},
\end{equation}
where $\mathrm{sr}(A)\triangleq \|A\|_{F}^{2}/\|A\|_{2}^{2}$ denotes the stable rank of $A$.
\end{lemma}

\begin{proof}
From lemma \ref{hanson_wright_inequality}, for all $p\geq 1$,
\begin{equation}
    \left\|\boldsymbol{z}^{T} A \boldsymbol{z}-\mathbb{E} \boldsymbol{z}^{T} A \boldsymbol{z}\right\|_{p} \leq \sqrt{p}\|A\|_{F}+p\|A\|_{2},
\end{equation}
where $\|\cdot\|_{p}$ denotes $(\mathbb{E}|\cdot|^{p})^{1/p}$. In particular, for $p=2$, we have
\begin{equation}
    \mathbb{E}|\boldsymbol{z}^{T} A \boldsymbol{z}-\mathbb{E} \boldsymbol{z}^{T} A \boldsymbol{z}|^{2} \leq (\sqrt{2}\|A\|_{F}+2\|A\|_{2})^{2},
\end{equation}
then based on the Chebyshev inequality \cite[Corollary 1.2.5]{Roman2018High},
\begin{linenomath}
\begin{equation}
    \label{one sample tail bound}
    \mathbb{P}\left\{ |\boldsymbol{z}^{T}A\boldsymbol{z}-{\mathrm{tr}}(A)|\geq \epsilon \|A\|_{F}\right\}\leq \frac{(\sqrt{2}\|A\|_{F}+2\|A\|_{2})^{2}}{\epsilon^{2}\|A\|_{F}^{2}}.
\end{equation}
\end{linenomath}
The proof ends by replacing $\|A\|_{F}^{2}/\|A\|_{2}^{2}$ with $\mathrm{sr}(A)$.\\
\begin{linenomath}
    \qed
\end{linenomath}
\end{proof}

\begin{theorem}
\label{the_hucthinson_estimation_error}
Let $A\in \mathbb{R}^{n\times n}$ be a symmetric matrix. Let $H^{N}(A)$ denote an $N$-query Hutchinson trace estimator implemented with Rademacher random vectors. For any given $\epsilon$ and $\delta$, if query number $N$ satisfies 
\begin{equation}
    \label{the_bound_of_N}
    N\geq \dfrac{\left(1+\sqrt{1+4\epsilon\sqrt{\delta}}\right)^{2}}{2\epsilon^{2}\delta},
\end{equation}
we have
\begin{equation}
    \mathbb{P}\left\{|H^{N}(A)-{\mathrm{tr}}(A)|\leq \epsilon \|A\|_{F}\right\}\geq 1-\delta.
\end{equation}
\end{theorem}

\begin{proof}
Define the following block diagonal matrix
\begin{linenomath}
    $$\mathcal{A}={\mathrm{diag}}(N^{-1}A,\ldots,N^{-1}A)\in \mathbb{R}^{N n\times N n},$$
\end{linenomath}
that is, matrix $\mathcal{A}$ consists of $N$ diagonal blocks containing rescaled copies of $A$.

The $N$-query Hutchinson trace estimator of $A$ equals $\tilde{z}^{T}\mathcal{A}\tilde{z}$ for a Rademacher vector $\tilde{z}$ of length $N\times n$, that is $\tilde{z}^{T}\mathcal{A}\tilde{z}=H^{N}(A)$. Note that ${\mathrm{tr}}(\mathcal{A})={\mathrm{tr}}(A)$, $\|\mathcal{A}\|_{F}=N^{-1/2}\|A\|_{F}$ and $\|\mathcal{A}\|_{2}=N^{-1}\|A\|_{2}$.
 
From lemma \ref{single_sample_inequality}, we have
\begin{linenomath}
    \begin{align*}
		\mathbb{P}\left\{|H^{N}(A)-{\mathrm{tr}}(A)|\geq \epsilon N^{-1/2}\|A\|_{F}\right\} &=\mathbb{P}\left\{|\tilde{z}^{T}\mathcal{A}\tilde{z}-\mathbb{E}\tilde{z}^{T}\mathcal{A}\tilde{z}|\geq \epsilon\|\mathcal{A}\|_{F}\right\}  \\
		&\leq \dfrac{(\sqrt{2}\|\mathcal{A}\|_{F}+2\|\mathcal{A}\|_{2})^{2}}{\epsilon^{2}\|\mathcal{A}\|_{F}^{2}}\\
		& = \dfrac{(\sqrt{2}+2(\mathrm{sr}(A) N)^{-\frac{1}{2}})^{2}}{\epsilon^{2}},
    \end{align*}
\end{linenomath}
that is 
\begin{equation}
    \mathbb{P}\left\{|H^{N}(A)-{\mathrm{tr}}(A)|\geq \epsilon\|A\|_{F}\right\} \leq \left(\frac{\sqrt{2}N^{\frac{1}{2}}+2\mathrm{sr}(A)^{-\frac{1}{2}}}{\epsilon N}\right)^{2}.
\end{equation}
Let the right-hand side of the above inequality be less than $\delta$, and as $\mathrm{sr}(A)>1$, for a given tolerance factor $\epsilon>0$, if the query number $N$ satisfies
\begin{linenomath}
    \begin{equation*}
    N \geq \frac{\left(1+\sqrt{1+4\epsilon\sqrt{\delta}}\right)^{2}}{2\epsilon^{2}\delta},
\end{equation*}
\end{linenomath}
we have
\begin{equation}
   \mathbb{P}\left\{|H^{N}(A)-{\mathrm{tr}}(A)|\leq \epsilon\|A\|_{F}\right\}\geq 1-\delta. 
\end{equation}
\qed
\end{proof}

Recall that $\Delta$ is an SPD matrix defined in (\ref{eq3}). Based on Theorem \ref{the_hucthinson_estimation_error}, if $N$ satisfies the condition formulated in (\ref{the_bound_of_N}), we have
\begin{equation}
    \label{the_error_bound_of_Delta}
    \mathbb{P}\left\{|H^{N}(\Delta)-\mathrm{tr}(\Delta)|\leq \epsilon \|\Delta\|_{F}\right\}\geq 1-\delta.
\end{equation}

\subsection{Hutchinson estimation error with bound of $tr(f(A))$}
\label{subsec:hutch2}
In this subsection, we will transform the error bound $\epsilon\|\Delta\|_{F}$ that appears in (\ref{the_error_bound_of_Delta}) into $\epsilon\cdot \mathrm{tr}(f(A))/2$ for the purpose of total error analysis.

Recall that $H^{N}(\Delta)$ is the Hutchinson trace estimator, as defined in (\ref{eq4}). If the number of query vectors $N$ satisfies the bound given in (\ref{the_bound_of_N}), then the following theorem holds.

\begin{theorem}
    \label{the_error_bound_of_part2}
    Given $\epsilon,\delta \in (0,1)$, $A\in \mathbb{R}^{n\times n}$ is an SPD matrix with minimum eigenvalue $\lambda_{\min}\geq 1$. Let $S\in \mathbb{R}^{n\times q}$ be a random matrix with independent normal random variables entries $S_{ij}\sim \mathcal{N}(0,1) /\sqrt{q}$, and $Q\in \mathbb{R}^{n\times k}$ consists of $k$-principal orthonormal bases of the column space spanned by $f(A)S$. If $k$ and $q$ satisfy
    \begin{linenomath}
        $$k\geq \dfrac{16(1+\epsilon)}{1-\epsilon}, q\geq \dfrac{288k}{\epsilon^{2}\delta},$$
    \end{linenomath}
    then
    \begin{equation}	
        \mathbb{P}\left\{|H^{N}(\Delta)-{\mathrm{tr}}(\Delta)|\leq 
		\frac{\epsilon}{4}{\mathrm{tr}}(f(A))\right\}\geq 1-\delta.
    \end{equation}
\end{theorem}

To prove this theorem, we first establish some properties of JL moments and PCPS, which were defined in Section 2.

\begin{lemma}
    \label{lemma_new1024}
    If matrix $S\in \mathbb{R}^{n\times q}$ satisfies the $(\frac{\epsilon}{6\sqrt{k}},\delta,\ell)$-JL moment property for any $\ell\geq2$, then with probability $\geq 1-4\delta$, $\tilde{A}=AS$ is an $(\epsilon,0,k)$-PCPS of $A$.
\end{lemma}

\begin{proof}
The proof of this lemma closely follows \cite[Lemma 6]{Cameron2020Proj}. The main difference is that the condition on S has been relaxed, resulting in a refined lower bound for the number of columns of the random matrix in \cite[Corollary7]{Cameron2020Proj}.

In this context, we mainly focus on the proof that if $S$ satisfies the $(\frac{\epsilon}{6\sqrt{k}},\delta, \ell)$-JL moment property for any $\ell\geq 2$, then for any $\boldsymbol{x}\in \mathbb{R}^{n}$,
\begin{equation}
    \label{subspace-embedding}
    \mathbb{P}\left\{\left|\|\boldsymbol{x}^{T}A_{k}\|_{2}^{2}-\|\boldsymbol{x}^{T}A_{k}S\|_{2}^{2}\right|>\frac{\epsilon}{3}\|\boldsymbol{x}^{T}A_{k}\|_{2}^{2}\right\}<\delta,
\end{equation}
where $A_{k}$ is the optimal rank-$k$ approximation of $A$. That is, $S$ is the $\frac{\epsilon}{3}$-subspace embedding for $A_{k}$ with probability $\geq 1-\delta$.

Let $\boldsymbol{y}^{T}=\boldsymbol{x}^{T}A_{k}\in \mathbb{R}^{n}$ and substitute it into of (\ref{subspace-embedding}),
\begin{equation}
    \label{new-subspace-embedding}
    \mathbb{P}\left\{\left|\|\boldsymbol{y}^{T}\|_{2}^{2}-\|\boldsymbol{y}^{T}S\|_{2}^{2}\right|>\frac{\epsilon}{3}\|\boldsymbol{y}^{T}\|_{2}^{2}\right\}< \delta, \forall \boldsymbol{y}\in {\mathrm{range}}(A_{k}^{T}),
\end{equation}
If $\boldsymbol{y}=\boldsymbol{0}$, the probability inequality (\ref{new-subspace-embedding}) is obvious. If $\boldsymbol{y}\neq \boldsymbol{0}$, we can normalize $\boldsymbol{y}$ by dividing its norm, that is,
\begin{equation}
    \label{goal-result}
    \mathbb{P}\left\{\left|\|\boldsymbol{y}^{T}S\|_{2}^{2}-1\right|>\frac{\epsilon}{3}\right\}< \delta, \forall \boldsymbol{y}\in {\mathrm{range}}(A_{k}^{T})\  \text{and}\  \|\boldsymbol{y}\|_{2}=1.
\end{equation}

Since $S$ satisfies the $(\frac{\epsilon}{6\sqrt{k}},\delta,\ell)$-JL moment property, we have
\begin{equation}
    \label{JL-moment}
    \mathbb{E}_{S}\left|\|\boldsymbol{y}^{T}S\|_{2}^{2}-1\right|^{\ell}\leq\left(\frac{\epsilon}{6\sqrt{k}}\right)^{\ell}\delta=\left(\frac{\epsilon}{3}\right)^{\ell}\left(\frac{1}{2\sqrt{k}}\right)^{\ell}\delta< \left(\frac{\epsilon}{3}\right)^{\ell}\delta.
\end{equation}

\noindent The second inequality comes from $(1/2\sqrt{k})^{\ell}\leq 1$ for all $k\geq 1, \ell\geq 2$.
Based on the Markov inequality, we have
\begin{equation}
    \label{chebyshev-inequality}
    \mathbb{P}\left\{\left|\|\boldsymbol{y}^{T}S\|_{2}^{2}-1\right|>\frac{\epsilon}{3}\right\}\leq \left(\frac{\epsilon}{3}\right)^{-\ell}\mathbb{E}_{S}\left|\|\boldsymbol{y}^{T}S\|_{2}^{2}-1\right|^{\ell}.
\end{equation}

\noindent Thus, the goal result in (\ref{goal-result}) can be derived by substituting (\ref{JL-moment})  into (\ref{chebyshev-inequality}). 

For the remainder of the proof of this lemma, please refer to \cite[Lemma 6]{Cameron2020Proj}.
\qed
\end{proof}

\begin{lemma}
    \label{the_bound_of_q}
    For any matrix $A\in \mathbb{R}^{n\times n}$, let $S\in \mathbb{R}^{n\times q}$ be a random matrix with independent normal random variables entries $S_{ij}\sim \mathcal{N}(0,1)/\sqrt{q}$. If $q\geq 288k/(\epsilon^{2}\delta)$, then $\tilde{A}=AS$ is an $(\epsilon,0,k)$-PCPS of $A$ with probability $\geq 1-\delta$.
\end{lemma}

\begin{proof} For any given $\boldsymbol{x}\in \mathbb{R}^{n}$ and $\|\boldsymbol{x}\|_{2}=1$, let $S_{*j}$ denote the $j$-th column of matrix $S$, then
\begin{linenomath}
    $$q\|\boldsymbol{x}^{T}S\|_{2}^{2}=q\sum_{j=1}^{q}(\boldsymbol{x}^{T}S_{*j})^{2}=\sum_{j=1}^{q}(\sum_{i=1}^{n}x_{i}\sqrt{q}S_{ij})^{2}.$$
\end{linenomath}
Let $y_{j}\triangleq \sum_{i=1}^{n}x_{i}\sqrt{q}S_{ij},j=1,2,\ldots,q$, as $\sqrt{q}S_{ij}\sim N(0,1)$ we have  $\mathbb{E}y_{j}=0$, $\mathbb{D}y_{j}=1$, that is, $y_{j}\sim N(0,1)$. As $q\|\boldsymbol{x}^{T}S\|_{2}^{2}=\sum_{j=1}^{q}y_{j}^{2}$, then random variable $Y\triangleq q\|\boldsymbol{x}^{T}S\|_{2}^{2}\sim \mathcal{X}^{2}(q)$.

As $\mathbb{E}Y = q$, $\mathbb{D}Y=2q$, thus
\begin{linenomath}
    $$\mathbb{E}|\|\boldsymbol{x}^{T}S\|_{2}^{2}-1|^{2}=\frac{2}{q}.$$
\end{linenomath}
When $q\geq \dfrac{288k}{\epsilon^{2}\delta}$,
\begin{linenomath}
    $$\mathbb{E}|\|\boldsymbol{x}^{T}S\|_{2}^{2}-1|^{2}\leq \left(\frac{\epsilon}{6\sqrt{k}}\right)^{2}\frac{\delta}{4},$$
\end{linenomath}
that is, $S$ satisfies $(\frac{\epsilon}{6\sqrt{k}},\frac{\delta}{4},2)$-JL moment property.

From Lemma \ref{lemma_new1024}, matrix $AS$ is an $(\epsilon,0,k)$-PCPS of $A$ with probability at least $1-\delta$. 
\qed
\end{proof}

\subsection{Proof of Theorem \ref{the_error_bound_of_part2}}
\begin{proof}
Based on Lemma \ref{the_bound_of_q}, if $q\geq 288k/(\epsilon^{2}\delta)$, then with probability not less than $1-\delta$, $f(A)S$ is an $(\epsilon,0,k)$-PCPS of $A$.

Let $\mathcal{P}$ be the set of rank-$k$ orthogonal projections. Let
\begin{linenomath}
    $$\tilde{P}^{*}\triangleq \arg \min_{P\in \mathcal{P}}\|f(A)S-Pf(A)S\|_{F}= QQ^{T},$$
\end{linenomath}
and 
\begin{linenomath}
    $$P^{*}\triangleq\arg\min_{P\in \mathcal{P}}\|f(A)-Pf(A)\|_{F}=U_{k}U_{k}^{T}.$$
\end{linenomath}
Based on the Definition \ref{projection_cost_preserving_sketch}, the following two inequalities hold:
\begin{equation}
    \label{one_of_the_two_ineq}
    (1-\epsilon)\|f(A)-\tilde{P}^{*}f(A)\|_{F}^{2}\leq \|f(A)S-\tilde{P}^{*}f(A)S\|_{F}^{2},
\end{equation}
\begin{equation}
    \label{another_of_the_two_ineq}
    \|f(A)S-P^{*}f(A)S\|_{F}^{2}\leq (1+\epsilon)\|f(A)-P^{*}f(A)\|_{F}^{2}.
\end{equation}
As $\|f(A)S-\tilde{P}^{*}f(A)S\|_{F}^{2}\leq \|f(A)S-P^{*}f(A)S\|_{F}^{2}$, combine (\ref{one_of_the_two_ineq}) and (\ref{another_of_the_two_ineq})
\begin{linenomath}
    $$\|f(A)-QQ^{T}f(A)\|_{F}^{2}\leq \frac{1+\epsilon}{1-\epsilon}\|f(A)-P^{*}f(A)\|_{F}^{2}.$$
\end{linenomath}
Furthermore, we have
\begin{linenomath}
    \begin{align*}
    \|\Delta\|_{F}&=\|(I-QQ^{T})f(A)(I-QQ^{T})\|_{F}\\
    &\leq \|I-QQ^{T}\|_{2} \|f(A)-QQ^{T}f(A)\|_{F}\\
    &=\|f(A)-QQ^{T}f(A)\|_{F}\\
    &\leq \frac{\sqrt{1+\epsilon}}{\sqrt{k(1-\epsilon)}}\mathrm{tr}(f(A)),
\end{align*}
\end{linenomath}
the first inequality is based on the sub-multiplicativity property of Frobenius norm, the second inequality results from Lemma \ref{k_rank_app}.

Combined the error bound formulated in (\ref{the_error_bound_of_Delta}), we have
\begin{equation}
    \label{error2}
    \mathbb{P}\left\{|H^{N}(\Delta)-{\mathrm{tr}}(\Delta)|\leq \frac{\epsilon\sqrt{1+\epsilon}}{\sqrt{k(1-\epsilon)}}{\mathrm{tr}}(f(A))\right\}\geq 1-\delta.
\end{equation}
Thus, when $k\geq 16(1+\epsilon)/(1-\epsilon)$,
\begin{equation}
    \mathbb{P}\left\{|H^{N}(\Delta)-{\mathrm{tr}}(\Delta)|\leq 
    \frac{\epsilon}{4}{\mathrm{tr}}(f(A))\right\}\geq 1-\delta.
\end{equation} 
\qed
\end{proof}

\begin{remark} 
    The result in \cite[Corollary7]{Cameron2020Proj} is of significant theoretical importance. However, the bound  given ($q\geq c\cdot (k+\log(1/\delta))/\epsilon^{2} $, where $c$ is a sufficiently large universal constant) lacks explicit constants, providing little guidance for practitioners. In the Appendix, Lemma \ref{loose-bound} complements the proof of \cite[Corollary7]{Cameron2020Proj} and provides an explicit bound. However, this bound is much looser than the one given in Theorem \ref{the_bound_of_q}. One reason is that the conditions on $S$ specified in \cite[Lemma 6]{Cameron2020Proj} are more stringent than those stated in Lemma \ref{lemma_new1024}.
\end{remark}

\begin{remark}
    In this subsection, we use PCPS tricks to analyze the error bound. An analysis without using PCPS tricks is also presented in the Appendix. A comparison of the performance of these two methods will be conducted in Section \ref{subsec:with and without pcps}.
\end{remark}

\subsection{The error bound of the first part}
\label{subsec:total}
Computing $f(A)$ explicitly requires a full eigendecomposition with $\Omega(n^{3})$ time complexity, which can be expensive for large values of $n$. As a result, directly computing $\mathrm{tr}(Q^{T}f(A)Q)$, the first part of the equation (\ref{eq3}), is usually not feasible. Instead, we can use the $(m+1)$-step Lanczos quadrature approximation method to approximate the computation of $\mathrm{tr}(Q^{T}f(A)Q)$.

\begin{equation}
    \label{tr_Q_A_Q}
    \mathrm{tr}(Q^{T}f(A)Q) = \sum_{i=1}^{k}(Q_{*i})^{T}f(A)Q_{*i}= \sum_{i=1}^{k}(U^{T}Q_{*i})^{T}f(\Lambda)U^{T}Q_{*i},
\end{equation}
where $Q_{*i}$ denotes the $i$-th column of $Q$, $\|U^{T}Q_{*i}\|_{2}=1$.

\begin{lemma}
    Given $\epsilon \in (0,1)$, $A\in \mathbb{R}^{n\times n}$ is an SPD matrix with minimum eigenvalue $\lambda_{\min}\geq 1$. Let $S\in \mathbb{R}^{n\times q}$ be a random matrix with independent normal random variables entries $S_{ij}\sim \mathcal{N}(0,1)/\sqrt{q}$, and $Q\in \mathbb{R}^{n\times k}$ consists of $k$-principal orthonormal bases of the column space spanned by $f(A)S$. Let $L_{m'+1}(Q^{T}f(A)Q)$ denote the $(m'+1)$-step Lanczos quadrature approximation of $\mathrm{tr}(Q^{T}f(A)Q)$. If
    \begin{linenomath}
        $$m'\geq \frac{1}{2}\log\left(\frac{2kC_{\rho}}{\varepsilon }\right)/\log(\rho),$$
    \end{linenomath}
    where $M_{\rho}=|\log(\lambda_{\min}/2)|+\pi$, $\rho=(\lambda_{\max}+\sqrt{2\lambda_{\max}\lambda_{\min}-\lambda_{\min}^{2}})/(\lambda_{\max}\lambda_{\min})$ and $C=(n-1)f(\lambda_{\min})+f(\lambda_{\max}), C_{\rho}=4M_{\rho}/(C(\rho^{2}-\rho))$,
    then
    \begin{linenomath}
        $$\left|\mathrm{tr}(Q^{T}f(A)Q)-L_{m'+1}(Q^{T}f(A)Q)\right| \leq \frac{\epsilon}{2}\mathrm{tr}(f(A)).$$
    \end{linenomath}
\end{lemma}

\begin{proof}
Let $\tilde{\boldsymbol{\mu}}^{(i)}=U^{T}Q_{*i}=[\tilde{\mu}_{1}^{(i)},\ldots,\tilde{\mu}_{n}^{(i)}]^{T}$ and substitute into (\ref{tr_Q_A_Q}),
we have
\begin{linenomath}
    \begin{align*}
        \mathrm{tr}(Q^{T}f(A)Q)&=\sum_{i=1}^{k}\sum_{j=1}^{n}f(\lambda_{j})(\tilde{\mu}_{j}^{(i)})^{2}\\
        &=\sum_{i=1}^{k}\int_{\lambda_{n}}^{\lambda_{1}}f(t)\mathrm{d}\tilde{\mu}^{(i)}(t)\\
         &\approx \sum_{i=1}^{k}\sum_{l=0}^{m'}\tilde{\tau}_{l}^{(i)}f(\tilde{\theta}_{l}^{(i)})\triangleq L_{m'+1}(Q^{T}f(A)Q),
    \end{align*}
\end{linenomath}
where $\{\tilde{\mu}^{(i)}(t), i=1,2,\ldots, k\}$ are measures of integration similar to (\ref{measure_function}),  and the $(m'+1)$ pairs of Gaussian nodes and weights are denoted by $\{(\tilde{\tau}_{l}^{(i)},\tilde{\theta}_{l}^{(i)})\}_{k=0}^{m'}$, which is consistent with  the description for (\ref{gauss_quadrature_approx}). 

Then the approximation error can be described  in the following form,
\begin{linenomath}
\begin{align}
    & \left|\mathrm{tr}(Q^{T}f(A)Q)-L_{m'+1}(Q^{T}f(A)Q)\right| \nonumber\\ 
    = & \left|\sum_{i=1}^{k}\left(\int_{\lambda_{n}}^{\lambda_{1}}f(t)\mathrm{d}\tilde{\mu}^{(i)}(t)- \sum_{l=0}^{m'}\tilde{\tau}_{l}^{(i)}f(\tilde{\theta}_{l}^{(i)})\right)\right| \nonumber\\ 
    \leq & \sum_{i=1}^{k}\left|\int_{\lambda_{n}}^{\lambda_{1}}f(t)\mathrm{d}\tilde{\mu}^{(i)}(t)- \sum_{l=0}^{m'}\tilde{\tau}_{l}^{(i)}f(\tilde{\theta}_{l}^{(i)})\right| \nonumber\\ 
    \leq & \frac{4kM_{\rho}}{(\rho-1)\rho^{2m'+1}}\leq \frac{\epsilon}{2}\mathrm{tr}(f(A)),  \label{the_first_error}
\end{align}
\end{linenomath}
where the second inequality results from Lemma \ref{lemma_1} and the last inequality is obtained by referring to the bound of $m$ as stated in (\ref{the_bound_of_m}) and satisfying
\begin{equation}
    m'\geq \frac{1}{2}\log\left(\frac{2kC_{\rho}}{\varepsilon }\right)/\log(\rho).
\end{equation}
\qed
\end{proof}

\subsection{Proof of Theorem \ref{the_main_result}}
\begin{proof}
Based on Theorem \ref{the_Lanczos_error} and Theorem \ref{the_error_bound_of_part2}, we have
\begin{linenomath}
    \begin{align*}
        1-\delta &\leq \mathbb{P}\left\{|H^{N}(\Delta)-{\mathrm{tr}}(\Delta)|\leq \frac{\varepsilon}{4}{\mathrm{tr}}(f(A))\right\}\\
    &\leq \mathbb{P}\left\{|H^{N}(\Delta)-{\mathrm{tr}}(\Delta)|+|H^{N}(\Delta)-L_{m+1}^{N}(\Delta)|\leq \frac{\varepsilon}{2}|\mathrm{tr}(f(A))|\right\}.
    \end{align*}
\end{linenomath}
By applying the triangle inequality, we have
\begin{equation}
    \label{the_second_error}
    \mathbb{P}\left\{|{\mathrm{tr}}(\Delta)-L_{m+1}^{N}(\Delta)|\leq \frac{\varepsilon}{2} |{\mathrm{tr}}(f(A))|\right\}\geq 1-\delta.
\end{equation}
From equation (\ref{eq3}), we can derive that 
${\mathrm{tr}}(f(A))={\mathrm{tr}}(Q^{T}f(A)Q)+{\mathrm{tr}}(\Delta)$. Let $\Gamma \triangleq L_{m'+1}(Q^{T}f(A)Q)+L_{m+1}^{N}(\Delta)$. By combining equations \eqref{the_first_error} and \eqref{the_second_error} and applying the triangle inequality again, we obtain the following equation,
\begin{linenomath}
    \begin{align*}
    \mathbb{P}\left\{ \right.& \left|{\mathrm{tr}}(f(A))-\Gamma\right|\\
    &\leq \left|\mathrm{tr}(Q^{T}f(A)Q)-L_{m'+1}(Q^{T}f(A)Q)\right|+\left|\mathrm{tr}(\Delta)-L_{m+1}^{N}(\Delta)\right| \\
    & \leq \epsilon \left|\mathrm{tr}(f(A))\right|\left.\right\}\geq 1-\delta.
    \end{align*}
\end{linenomath}
This completes the proof of Theorem 3.1. 
\qed
\end{proof}

\section{Numerical results}
\label{sec:performance}
In this section, we introduce and evaluate our proposed algorithm for estimating the log-determinant of a positive definite matrix $A\in \mathbb{R}^{n\times n}$. The pseudocode description of the algorithm is as follows, where Algorithm \ref{algorithm2} provides an estimation of $\mathrm{tr}(Q^{T}f(A)Q)$, while Algorithm \ref{algorithm2} returns the final estimation of $\log\det(A)$.
\begin{algorithm}[b!]
    \caption{Log-determinant estimation with Lanczos Quadrature Algorithm}
    \label{algorithm1}
    \begin{algorithmic}[1]
        \Require{PSD matrix $A\in \mathbb{R}^{n\times n}$, function $f$, rank of orthonormal subspace $k$. Lanczos iteration steps $m$. The number of Hutchinson query vectors $N$. The number of columns of the random matrix $q$.}
        \Require{$S\in \mathbb{R}^{n\times q}$ with i.i.d. $N(0,1)/\sqrt{q}$ Gaussian random entries.}  
        \Require{Orthonormal basis $Q\in \mathbb{R}^{n\times k}$ for the span of $f(A)S$.}
        \Require{$L_{m'+1}$, the estimation of $\mathrm{tr}(Q^{T}f(A)Q)$, returned by Algorithm \ref{algorithm2}.}
        \Ensure{Return $L_{m'+1}+\frac{1}{N}\tilde{L}_{m+1}^{N}$}
        \For{$i=1$ to $N$}
            \State{Generate a Rademacher random vector $\boldsymbol{z}_{i}$}
            \State{Form unit vector $\boldsymbol{v}_{i}=(I-QQ^{T})\boldsymbol{z}_{i}/\|(I-QQ^{T})\boldsymbol{z}{i}\|_{2}$}
            \State{$T={\mathrm{Lanczos}}(A,\boldsymbol{v}_{i},m+1)$ as \cite{Golub2009Matrices}}
            \State{$[Y,\Theta]={\mathrm{eig}}(T)$ and compute $\tau_{k}=[\boldsymbol{e}_{1}^{T}\boldsymbol{y}_{k}]^{2} $ for $k=0,1,\ldots,m$}
            \State{$\tilde{L}_{m+1}^{N}\leftarrow \tilde{L}_{m+1}^{N}+\|(I-QQ^{T})\boldsymbol{z}_{i}\|_{2}^{2}\sum_{k=0}^{m}\tau_{k}f(\theta_{k})$}
        \EndFor
    \end{algorithmic}
\end{algorithm}

\begin{algorithm}
    \caption{Approximate $\mathrm{tr}(Q^{T}f(A)Q)$ by $L_{m'+1}$}
    \label{algorithm2}
    \begin{algorithmic}[1]
        \Require{PSD matrix $A\in \mathbb{R}^{n\times n}$, function $f$, column orthonormal matrix $Q\in \mathbb{R}^{n\times k}$, Lanczos iteration steps $m'$.}
        \Ensure{Return $L_{m'+1}$}
        \For{$i=1$ to $k$}
        \State{Form unit vector $\boldsymbol{v}_{i}=Q(:,i)$}
        \State{$T={\mathrm{Lanczos}}(A,\boldsymbol{v}_{i},m'+1)$ as \cite{Golub2009Matrices}}
        \State{$[Y,\Theta]={\mathrm{eig}}(T)$ and compute $\tau_{k}=[\boldsymbol{e}_{1}^{T}\boldsymbol{y}_{k}]^{2} $ for $k=0,1,\ldots,m'$}
        \State{$L_{m'+1}\leftarrow L_{m'+1}+\sum_{k=0}^{m'}\tau_{k}f(\theta_{k})$}
        \EndFor
    \end{algorithmic}
\end{algorithm}

Next, we will evaluate the methods proposed in this text from three perspectives: 1) A performance comparison between the method with PCPS proposed in the main text and the method without PCPS provided in the appendix; 2) A performance comparison with methods proposed in related literature; and 3) A comparison between the experimental and theoretical parameter values when the algorithm reaches a specified probability error bound for a given matrix.

In the following numerical experiments, we set $\delta=0.1$ and uniformly sample the value of parameter $\epsilon$ in the interval $[0.01,0.2]$. The symmetric positive definite matrix $A$ used for algorithm evaluation is set in the following form:
\begin{equation}
    \label{example_matrix}
    A=I+\sum_{j=1}^{40}\frac{10}{j^{2}}\boldsymbol{x}_{j}\boldsymbol{x}_{j}^{T}+\sum_{j=41}^{300}\frac{1}{j^{2}}\boldsymbol{x}_{j}\boldsymbol{x}_{j}^{T},
\end{equation}
where $\boldsymbol{x}_{1},\ldots,\boldsymbol{x}_{300}\in \mathbb{R}^{5000}$ are generated in Matlab using $\textbf{sprandn}(5000,1,0.025)$. This example comes from \cite{Saibaba2017Random,Alice2021On,persson2022improved} and satisfies the condition that the minimum eigenvalue is not less than $1$. In order to facilitate the reproduction of the experimental results in this paper, we set the seed of $\textbf{rng(seed)}$ to 50.

\subsection{Performance comparison of methods with and without PCPS}
\label{subsec:with and without pcps}
The bounds of the algorithm-related parameters were provided in the previous text. The derivation and analysis of these bounds were combined with PCPS tricks. As mentioned earlier, another method that does not use PCPS is also provided in the appendix. Table \ref{table_1} lists all the designed parameters. Since the use of PCPS tricks does not affect the estimation of the trace in the second part, the relevant parameters ($m$ and $N$) are consistent in both methods.

Therefore, it is sufficient to compare only the relevant parameters of the first part and the number of MVM required for the first part. Figure \ref{with_and_without} (a) shows the relationship between the number of columns in the random matrix $S$ and the relative error tolerance parameter $\epsilon$. In the method without PCPS, the number of columns in the random matrix is only related to probability parameter $\delta$ and does not change with the changes in $\epsilon$. Figure \ref{with_and_without} (b) shows that the number of Lanczos iteration steps required by the method with PCPS in the first part is significantly less than that required by the method without PCPS.  The number of MVM required by the method with PCPS for the first part is $q+km'$, while for the method without PCPS, the number of MVM is $(k+p)(1+m')$. And the comparison of the numerical values is described in Figure \ref{with_and_without} (c).

\begin{table}[h]
    \caption{Comparison of parameters between methods with and without PCPS.}
    \label{table_1}%
    \begin{tabular}{@{}lll@{}}
        \toprule
        parameters & \bf with PCPS &  \bf without PCPS\\
        \midrule
        columns of Q & $k: k\geq \frac{16(1+\epsilon)}{1-\epsilon}$& $k+p: k\geq \frac{64}{\delta^{2}},p\geq 1+\frac{64k}{k\delta^{2}-64}$\\
		columns of S & $q: q\geq \frac{288k}{\epsilon^{2}\delta}$ & $k+p: k\geq \frac{64}{\delta^{2}},p\geq 1+\frac{64k}{k\delta^{2}-64}$\\
		Lanczos steps of first part & $m'\geq \frac{\log \frac{2kC_{\rho}}{\epsilon}}{2\log \rho}$ & $m'\geq \frac{\log \frac{2(k+p)C_{\rho}}{\epsilon}}{2\log \rho}$\\
        Lanczos steps of second part & \multicolumn{2}{c}{$m\geq \frac{1}{2}\log \left(\frac{4nC_{\rho}}{\epsilon }\right)/\log \rho$} \\
        Hutchinson queries & \multicolumn{2}{c}{$N\geq \frac{1}{2}(1+\sqrt{1+4\epsilon\sqrt{\delta}})^{2}/(\epsilon^{2}\delta)$} \\
        \bottomrule
    \end{tabular}
\end{table}

\begin{figure}[htbp]
    \subfigure[]{
    \begin{minipage}[t]{0.32\linewidth}
    \includegraphics[width=1\textwidth]{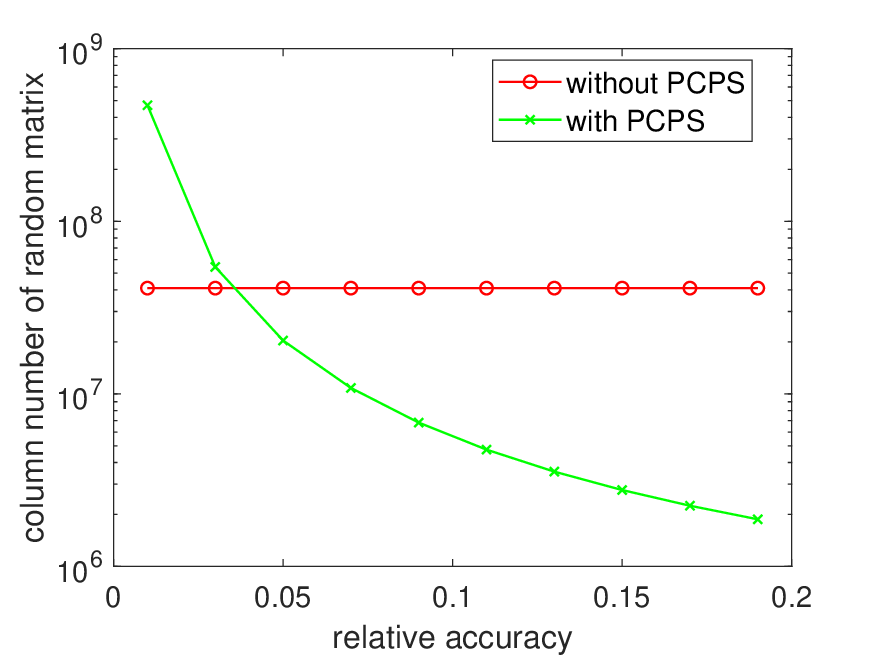}
    \end{minipage}%
    }%
    \subfigure[]{
    \begin{minipage}[t]{0.32\linewidth}
    \includegraphics[width=1\textwidth]{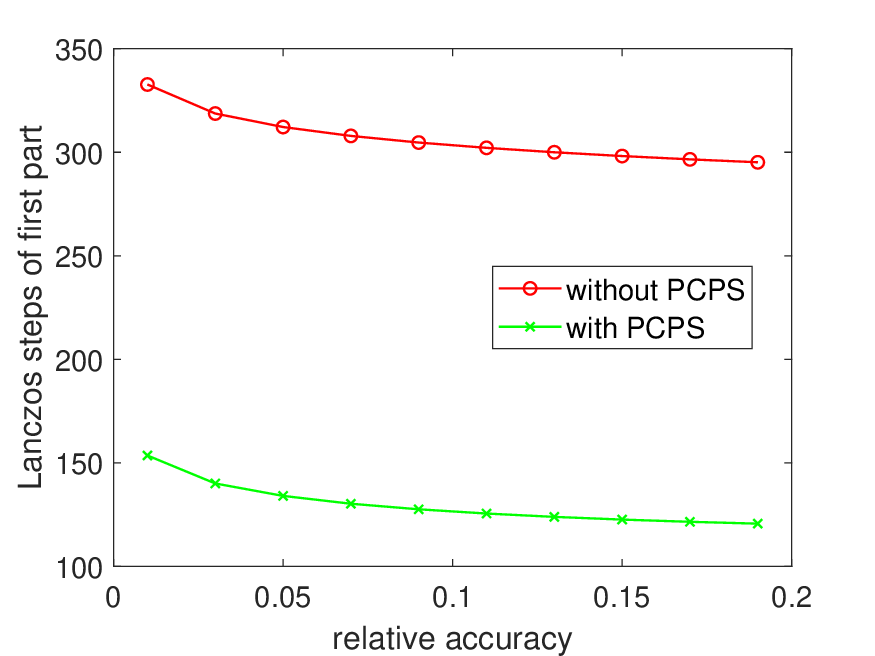}
    \end{minipage}%
    }%
    \subfigure[]{
    \begin{minipage}[t]{0.32\linewidth}
    \includegraphics[width=1\textwidth]{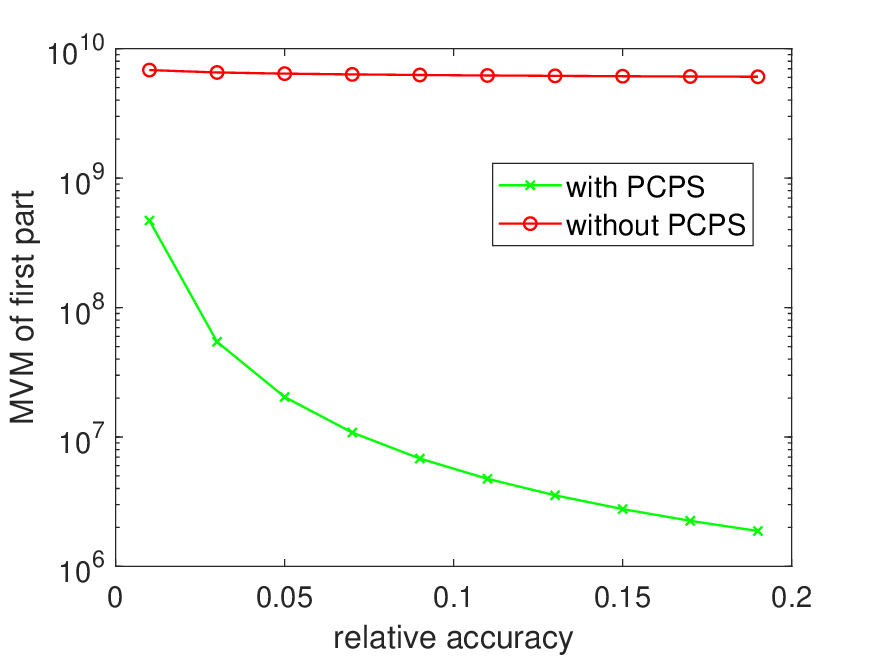}
    \end{minipage}
    }%
\centering
\caption{ Performance evaluations of methods with PCPS and without PCPS: (a) the column number of the random matrix $S$ vs. relative accuracy $\epsilon$, (b) Lanczos steps of the first part vs. relative accuracy $\epsilon$, (c) the total matrix-vector multiplications (MVM) of the first part vs. relative accuracy $\epsilon$.}
\label{with_and_without}
\end{figure}

\subsection{Comparison with other methods}
For the same matrix $A$ setting, we use the symbol $\Gamma_{(N,m)}$ to represent the estimation result of the estimator proposed in \cite{Ubaru2017Fast} and \cite{Alice2021On}, with $N$ being the number of query vectors and $m$ being the number of Lanczos steps.

To satisfy 
\begin{equation}
    \label{probability_inequality_of_Delta}
    \mathbb{P}\left\{\left|\mathrm{tr}(f(A))-\Gamma_{(N,m)}\right|\leq \epsilon|\mathrm{tr}(f(A))|\right\}\geq 1-\delta,
\end{equation}
it is shown in \cite[Corollary 4.5]{Ubaru2017Fast} that $N$ and $m$ should satisfy
\begin{equation}
    \label{the_bound_of_N_in_SLQ}
    N\geq \frac{24n^{2}}{\left(\epsilon \mathrm{tr}(f(A))\right)^{2}}\log^{2}(1+\kappa)\log\left(\frac{2}{\delta}\right),
\end{equation}
\begin{equation}
    m\geq \frac{\sqrt{3\kappa}}{4}\log\left(\frac{5\kappa n\log(2\kappa+2)}{\epsilon \mathrm{tr}(f(A))\sqrt{2\kappa+1}}\right),
\end{equation}
where $\kappa$ denotes the condition number of $A$. 

While, \cite[lemma 6] {Alice2021On} shows that the probability inequality in (\ref{probability_inequality_of_Delta}) holds when $N$ satisfies
\begin{equation}
    N\geq \frac{8}{\left(\epsilon \mathrm{tr}(f(A))\right)^{2}}(n\log^{2}\kappa +2\epsilon \mathrm{tr}(f(A))\log \kappa)\log\left(\frac{2}{\delta}\right),
\end{equation}
\begin{equation}
    m\geq \frac{\sqrt{\kappa+1}}{4}\log\left(\frac{4n(\sqrt{\kappa+1}+1)\log(2\kappa)}{\epsilon \mathrm{tr}(f(A))}\right).
\end{equation}
 The new bound reduces the number of required query vectors by a factor $n$ compared  to that in (\ref{the_bound_of_N_in_SLQ}).   

As shown in Tables \ref{table_2} and \ref{table_3}, compared to our lower bound presented in equation (\ref{the_bound_of_N}), the other two bounds require more matrix information and depend on the matrix dimension $n$, which will become very large as $n$ increases.

\begin{table}[t]
\caption{Comparison of the number of query vectors $N$ for different estimators.}\label{table_2}%
\begin{tabular}{@{}ll@{}}
\toprule
($\epsilon,\delta$) estimator & \bf lower bound of Rademacher queries $N$  \\
\midrule
\cite[Corollary 4.5]{Ubaru2017Fast} & $\frac{24}{\left(\epsilon \mathrm{tr}(f(A))\right)^{2}}n^{2}\log^{2}(1+\kappa)\log(\frac{2}{\delta})$ \\
\cite[lemma 6] {Alice2021On} & $\frac{8}{\left(\epsilon \mathrm{tr}(f(A))\right)^{2}}(n\log^{2}\kappa+2\epsilon \mathrm{tr}(f(A))
		\log\kappa)\log(\frac{2}{\delta})$  \\
this paper & $ \left(1+\sqrt{1+4\epsilon\sqrt{\delta}}\right)^{2}/(2\epsilon^{2}\delta)$ \\
\bottomrule
\end{tabular}
\end{table}

\begin{table}[t]
    \caption{Comparison of the Lanczos steps $m$ for different estimators.}
    \label{table_3}%
    \begin{tabular}{@{}ll@{}}
        \toprule
        ($\epsilon,\delta$) estimator & \bf lower bound of Lanczos steps $m$  \\
        \midrule
        \cite[Corollary 4.5]{Ubaru2017Fast} & $\frac{\sqrt{3\kappa}}{4}\log\left(\frac{5\kappa n\log(2\kappa+2)}{\epsilon \mathrm{tr}(f(A))\sqrt{2\kappa+1}}\right)$\\
		\cite[lemma 6] {Alice2021On} & $\frac{\sqrt{\kappa+1}}{4}\log\left(\frac{4n(\sqrt{\kappa+1}+1)\log(2\kappa)}{\epsilon \mathrm{tr}(f(A))}\right)$\\
        the first part &  $\frac{1}{2}\log \left(\frac{2kC_{\rho}}{\epsilon }\right)/\log \rho $ \\
        the second part & $ \frac{1}{2}\log\left(\frac{4nC_{\rho}}{\epsilon }\right)/\log \rho$ \\ 
        \bottomrule
    \end{tabular}
\end{table}

The bound of $N$ presented in \cite[lemma 6] {Alice2021On} is a simplified version of the result in \cite[Theorem 5] {Alice2021On}. We compare both in Figure \ref{compare_experiments}. As shown in Figure \ref{compare_experiments} (a), the conclusion in  \cite[Theorem 5] {Alice2021On} has the smallest sampling bound, followed by the conclusion in this paper. Since we divide the log-determinant problem into two parts, the Lanczos steps required for both parts are shown in Figure \ref{compare_experiments} (b). The comparison of the total MVM required for algorithm implementation is shown in Figure \ref{compare_experiments} (c), with \cite[Theorem 5] {Alice2021On} having the least MVM calculation, followed by this article. However, it should be noted that the result in \cite[Theorem 5] {Alice2021On} involves calculating the spectral norm and F-norm of matrix $(\log(A)-D_{\log(A)})$, where $D_{\log(A)}$ denotes the diagonal matrix containing the diagonal entries of $\log(A)$, which can be difficult to obtain for large-scale matrices.

\begin{figure}[htbp]
    \centering
    \subfigure[]{
    \begin{minipage}[t]{0.32\linewidth}
    \centering
    \includegraphics[width=1\textwidth]{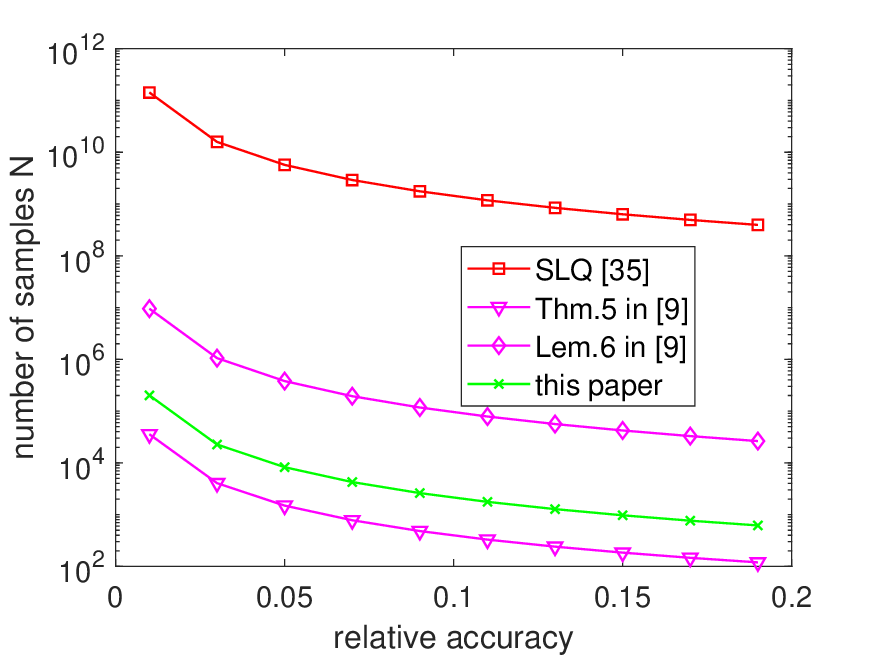}
    \end{minipage}%
    }%
    \subfigure[]{
    \begin{minipage}[t]{0.32\linewidth}
    \centering
    \includegraphics[width=1\textwidth]{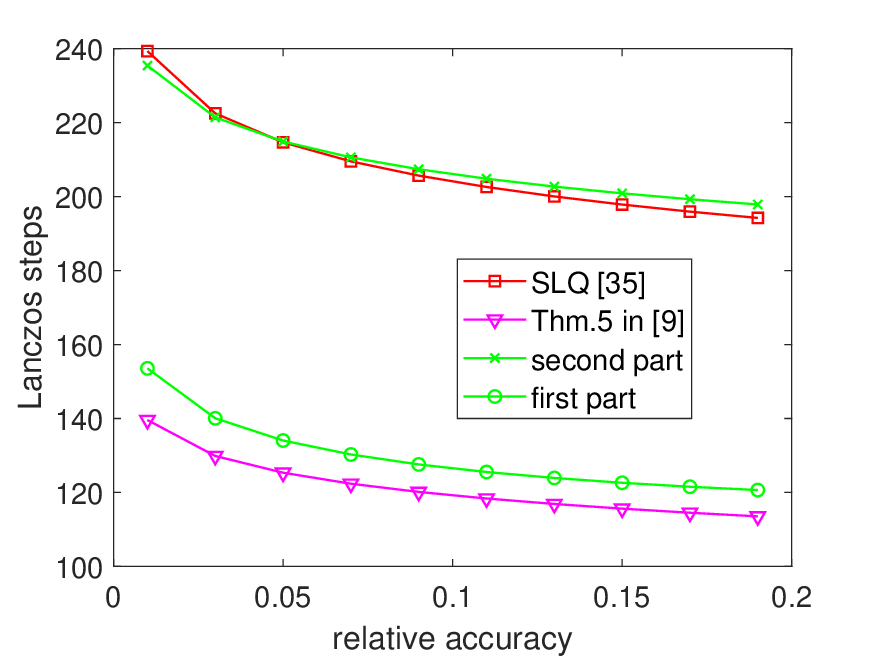}
    \end{minipage}%
    }%
    \subfigure[]{
    \begin{minipage}[t]{0.32\linewidth}
    \centering
    \includegraphics[width=1\textwidth]{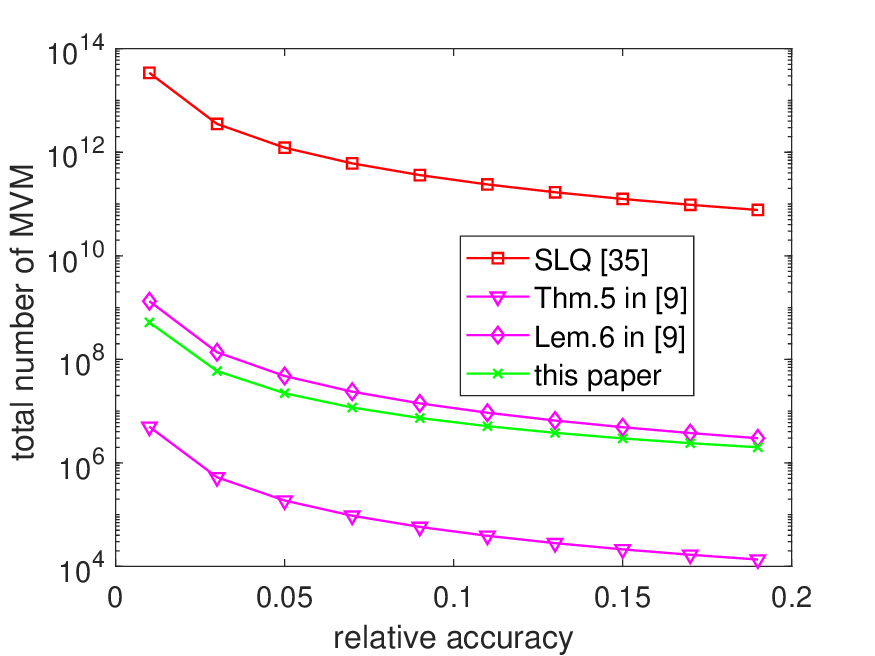}
    \end{minipage}
    }%
    \centering
    \caption{Performance evaluations of Algorithm \ref{algorithm1} and comparisons with \cite{Ubaru2017Fast} and \cite{Alice2021On}: (a) the number of Hutchinson query vectors vs. relative accuracy $\epsilon$, (b) Lanczos steps vs. relative accuracy $\epsilon$, (c) the total matrix-vector multiplication (MVM) vs. relative accuracy $\epsilon$.}
    \label{compare_experiments}
\end{figure}

\subsection{Theoretical and experimental values}
In this section, we use an example to verify the theoretical bounds given in the previous text. As shown in Table \ref{table_1}, the bound on the number of columns in $Q$ is relatively tight. However, the bound on the number of columns in the random matrix $S$ is not as tight. In our experiments, we observed that $q\approx 3k$ is sufficient for achieving high accuracy.

For the matrix presented in equation (\ref{example_matrix}), Figure \ref{theorem_experiments} shows that the experimental results perform much better than that predicted by our theoretical bounds. The number of Lanczos steps required for the first and second parts are depicted in Figure \ref{theorem_experiments} (a) and (b) respectively. Figure \ref{theorem_experiments} (c) shows the average number of MVM required by our algorithm over $100$ runs, with the shaded green area representing the $10$th to $90$th percentiles of these results. 

\begin{figure}[htbp]
    \centering
    \subfigure[]{
    \begin{minipage}[t]{0.32\linewidth}
    \centering
    \includegraphics[width=1\textwidth]{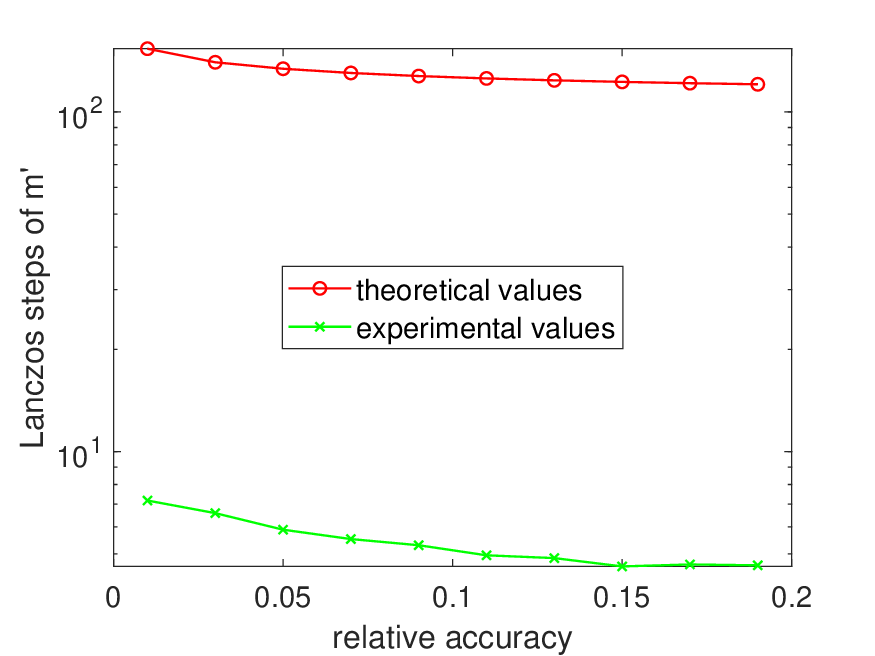}
    \end{minipage}%
    }%
    \subfigure[]{
    \begin{minipage}[t]{0.32\linewidth}
    \centering
    \includegraphics[width=1\textwidth]{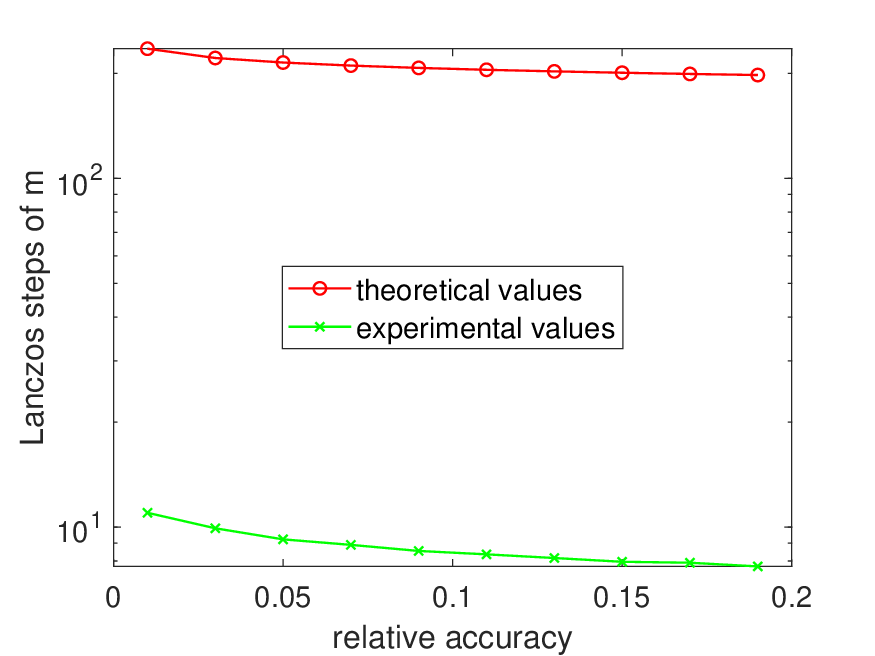}
    \end{minipage}%
    }%
    \subfigure[]{
    \begin{minipage}[t]{0.32\linewidth}
    \centering
    \includegraphics[width=1\textwidth]{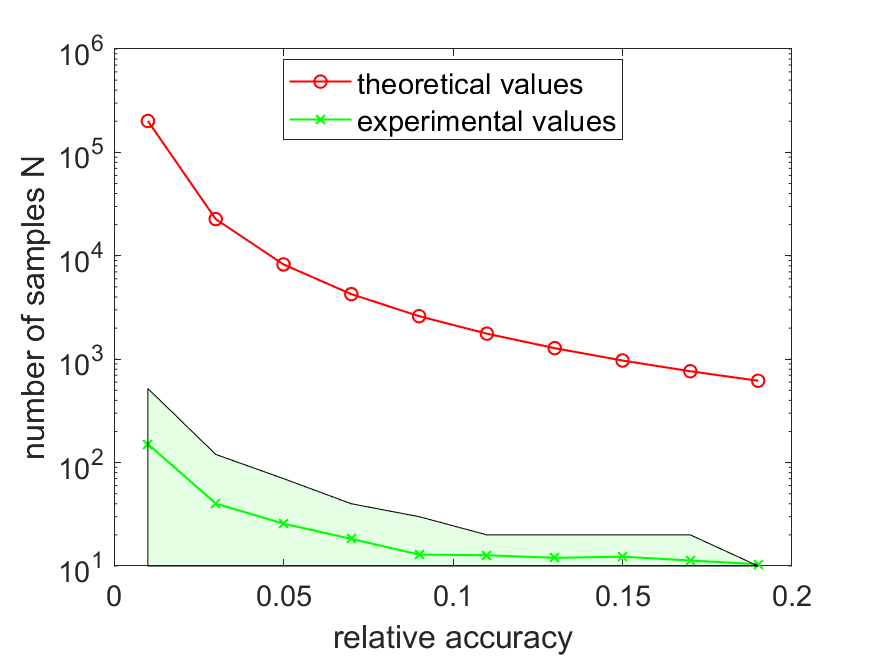}
    \end{minipage}
    }%
    \centering
    \caption{Comparison of theoretical and experimental values in a specific example: (a) Lanczos steps of the first part, (b) Lanczos steps of the second part, (c) the Hutchinson query vectors needed in the second part.}
    \label{theorem_experiments}
\end{figure}

\section{Conclusions}
\label{sec:conclusions}
In this paper, we analyze the error for approximating the log-determinant of large-scale positive definite matrices, where subspace projection and the stochastic Lanczos quadrature method are used. We provide a deterministic bound for the trace approximation of the projection part and a probabilistic bound for the remaining part. Unlike most literature that gives asymptotic upper or lower bounds, this paper presents explicit bounds for all of the algorithm-related design parameters.

Although the explicit bound for the number of columns in the random matrix $S$ may appear large, it is sufficient to guarantee the $(\epsilon,\delta)$ probability error bound for any large-scale matrix. Further research will also be conducted to explore the optimal bounds, filling the gap between the theoretical bounds and the actual requirements. Besides, this paper adopts a fixed-ratio error-bound allocation scheme, while we also present an optimized error allocation technique in \cite{li2023analysis} to reduce the overall MVM required by SLQ.

Note that this paper focuses on the computation of $\mathrm{tr}(f(A))$, where $f(\cdot)=\log(\cdot)$. However, the techniques and results presented in this paper have significant theoretical and computational implications for the trace estimation of other matrix functions, which would help solve those corresponding practical applications efficiently.

\begin{acknowledgements}
This work was funded by the natural science foundation of China (12271047); Guangdong Provincial Key Laboratory of Interdisciplinary Research and Application for Data Science, BNU-HKBU United International College (2022B1212010006); UIC research grant(R0400001-22; UICR0400008-21; R72021114); Guangdong College Enhancement and Innovation Program(2021ZDZX1046).
\end{acknowledgements}

\begin{appendices}
\section{Appendices}
\subsection{A looser bound of column number $q$}
\begin{lemma}
    \label{loose-bound}
    For any matrix $A\in \mathbb{R}^{n\times n}$, let $S\in \mathbb{R}^{n\times q}$ be a random matrix with entries $S_{ij} \sim \mathcal{N}(0,1)/\sqrt{q}$ are independent standard normal random variables. If $q \geq 1152e \frac{9^{k}}{\epsilon^{2}\delta}$, $\tilde{A}=AS$ is an $(\epsilon,0,k)$-PCPS of $A$ with probability $\geq 1-\delta$.
\end{lemma}

\begin{proof} For any given $\boldsymbol{x}\in \mathbb{R}^{n}$ and $\|\boldsymbol{x}\|_{2}=1$, then random variable $Y\triangleq q\|\boldsymbol{x}^{T}S\|_{2}^{2}\sim \mathcal{X}^{2}(q)$ (as explained in lemma \ref{the_bound_of_q}).

Let $\tilde{Y}=Y-q$, then $\mathbb{E}\tilde{Y}=0$ and the moment generating function (MGF) of $\tilde{Y}$ is
\begin{linenomath}
    \begin{align*}
    \mathbb{E}\exp(\lambda \tilde{Y}) & = \mathbb{E}\exp(\lambda(Y-q))\\
    &=\mathbb{E}\exp(\lambda Y)\exp(-\lambda q)\\
    &=(1-2\lambda)^{-\frac{q}{2}}\exp(-\lambda q)\\
    &\leq \exp(4q\lambda^{2}), \text{for} |\lambda|\leq \frac{1}{\sqrt{5}}
\end{align*}
\end{linenomath}
the third equality is derived by the MGF of $\mathcal{X}^{2}(q)$.

Let $K=2\sqrt{q}$, then for all $\lambda$ if $|\lambda|\leq 1/K,(q\geq 2)$,
\begin{equation}
	\label{exp_eq}
	\mathbb{E}\exp(\lambda \tilde{Y})\leq \exp(K^{2}\lambda^{2}).
\end{equation}
For all $x\in \mathbb{R}$ and $r>0$, the following inequality is valid \cite{Roman2018High}
\begin{linenomath}
    $$|x|^{r}\leq r^{r}(\exp(x)+\exp(-x)).$$
\end{linenomath}
Substituting $x=\tilde{Y}/K$ and taking expectation, we have
\begin{linenomath}
\begin{align*}
    \mathbb{E}| \tilde{Y}/K|^{r}&\leq r^{r}(\mathbb{E}\exp( \tilde{Y}/K)+\mathbb{E}\exp(-\tilde{Y}/K))\\
    &\leq 2er^{r}
\end{align*}
\end{linenomath}
the second inequality holds by sustituting $\lambda=1/K$ into equation (\ref{exp_eq}).
Then 
\begin{linenomath}
\begin{align*}
    \mathbb{E}|\|\boldsymbol{x}^{T}S\|_{2}^{2}-1|^{r}\leq 2e(Kr/q)^{r}.
\end{align*}
\end{linenomath}
Specially, take $r=2$, if $q \geq 1152e \frac{9^{k}}{\epsilon^{2}\delta}$, it is easy verify the following inequality holds.
\begin{linenomath}
\begin{align*}
    2e(\frac{4}{\sqrt{q}})^{2}\leq (\frac{\epsilon}{3})^{2}\cdot \frac{\delta}{4}\cdot \min\{
    (\frac{1}{2\sqrt{k}})^{2},\frac{1}{9^{k}}\}\\
\end{align*}
\end{linenomath}
That is, $S$ satisfies both the $(\frac{\epsilon}{6\sqrt{k}},\frac{\delta}{4},2)$-JL moment property and the $(\frac{\epsilon}{3},\frac{\delta/4}{9^{k}},2)$-JL moment property. From Lemma 6 in \cite{Cameron2020Proj}, with probability at least $1-\delta$, $AS$ is an $(\epsilon,0,k)$-projection-cost-perserving sketch of $A$.
\qed
\end{proof}

\subsection{Without using PCPS}
In this section, we will provide a new analytical route for the probabilistic error bounds
\begin{linenomath}
    $$\mathbb{P}\left\{\left|H^{N}(\Delta)-\mathrm{tr}(\Delta)\right|\leq \frac{\epsilon}{4}\mathrm{tr}(f(A))\right\}\geq 1-\delta.$$
\end{linenomath}
Slightly different from (\ref{eq3}) with the projection matrix $Q\in \mathbb{R}^{n\times k}$, we have the following new decomposition form of $\mathrm{tr}(f(A))$,
\begin{equation}
    \begin{aligned}
	\mathrm{tr}(f(A))&=\mathrm{tr}(Q^{T}f(A)Q)+\mathrm{tr}((I-QQ^{T})f(A)(I-QQ^{T}))\\
	&= \mathrm{tr}(Q^{T}f(A)Q)+\mathrm{tr}(\tilde{\Delta}).
    \end{aligned}
\end{equation}
Let $\tilde{\Delta}\triangleq (I-QQ^{T})f(A)(I-QQ^{T})$, and $Q\in \mathbb{R}^{n\times (k+p)}$ is the $k+p$ orthogonal basis of the span of $f(A)S$, $S\in \mathbb{R}^{n\times (k+p)}$ is an standard Gaussian random matrix.

Similar to Theorem \ref{the_error_bound_of_part2}, we have the following lemma.
\begin{lemma}
    Given $\epsilon,\delta \in (0,1)$, an SPD matrix $A\in \mathbb{R}^{n\times n}$ with its minimum eigenvalue $\lambda_{\min}\geq 1$. Let $S\in \mathbb{R}^{n\times (k+p)}$ be a standard Gaussian matrix, and $Q\in \mathbb{R}^{n\times (k+p)}$ whose columns form an orthonormal basis of the range of $f(A)S$. If the target rank $k$ and oversampling parameter $p$ satisfy
    \begin{linenomath}
        $$k\geq \frac{64}{\delta^{2}}, p\geq 1+\frac{64k}{k\delta^{2}-64},$$
    \end{linenomath}
    then 
    \begin{equation}
		\mathbb{P}\left\{\left|H^{N}(\tilde{\Delta})-\mathrm{tr}(\tilde{\Delta})\right|\leq \frac{\epsilon}{4}\mathrm{tr}(f(A))\right\}\geq 1-\delta.
    \end{equation}
\end{lemma}

\begin{proof} 
Based on Theorem 10.5 in \cite{Halko2011Finding}, we have
\begin{equation}
    \mathbb{E}\|f(A)-QQ^{T}f(A)\|_{F}\leq \left(1+\frac{k}{p-1}\right)^{1/2}\|f(A)-A_{k}(f)\|_{F}.
\end{equation}
Using Markov inequality, we have
\begin{equation}
    \mathbb{P}\left\{\|f(A)-QQ^{T}f(A)\|_{F}\leq \frac{\sqrt{k}}{4}\|f(A)-A_{k}(f)\|_{F}\right\}\geq 1-\frac{(1+\frac{k}{p-1})^{1/2}}{\sqrt{k}/4},
\end{equation}
where $k\geq 16$, $p\geq (17k-16)/(k-16)$.

From lemma \ref{k_rank_app}, we have $\|f(A)-A_{k}(f)\|_{F}\leq \mathrm{tr}(f(A))/\sqrt{k}$. Then
\begin{equation}
    \mathbb{P}\left\{\|f(A)-QQ^{T}f(A)\|_{F}\leq \frac{1}{4}\mathrm{tr}(f(A))\right\}\geq 1-\frac{(1+\frac{k}{p-1})^{1/2}}{\sqrt{k}/4}.
\end{equation}
And when the oversampling number $p$ satisfies
\begin{equation}
    p \geq 1+\frac{64k}{k\delta^{2}-64},
\end{equation}
where $k\geq 64/\delta^{2}$, we have
\begin{equation}
    \mathbb{P}\left\{\|f(A)-QQ^{T}f(A)\|_{F}\leq \frac{1}{4}\mathrm{tr}(f(A))\right\}\geq 1-\frac{\delta}{2}.
\end{equation}
As $\|\tilde{\Delta}\|_{F}\leq \|f(A)-QQ^{T}f(A)\|_{F}$. Based on Theorem \ref{the_hucthinson_estimation_error}, the following formulation is hold 
\begin{equation}
    \mathbb{P}\left\{\left|H^{N}(\tilde{\Delta})-\mathrm{tr}(\tilde{\Delta})\right|\leq \frac{\epsilon}{4}\mathrm{tr}(\log(A))\right\}\geq 1-\delta.
\end{equation}
with the number of columns of the Gaussian random matrix $S$ satisfies
\begin{equation}
    (k+p) \geq \frac{(k^{2}+k)\delta^{2}-64}{k\delta^{2}-64}.
\end{equation}
\qed
\end{proof}




\end{appendices}


\bibliographystyle{spmpsci}      
\bibliography{References}

\end{document}